\pdfoutput=1
\documentclass[oneside,a4paper,12pt,reqno]{amsart}
\usepackage{typearea}
\typearea{13}

\usepackage[utf8]{inputenc}
\usepackage[UKenglish]{babel}
\usepackage[T1]{fontenc}

\usepackage{amsmath}
\usepackage{amsthm}
\usepackage{amsfonts}
\usepackage{amssymb}
\usepackage{enumitem}

\usepackage{dsfont}
\usepackage{ifthen}
\usepackage{mathtools}
\usepackage{mathrsfs}
\usepackage{etoolbox}
\usepackage{xparse}

\usepackage{microtype}
\usepackage{url}
\usepackage[plainpages=false,colorlinks=false,unicode]{hyperref}

\mathtoolsset{centercolon}
\numberwithin{equation}{section}

\theoremstyle{plain}
\newtheorem{thm}{Theorem}[section]
\newtheorem{prop}[thm]{Proposition}
\newtheorem{cor}[thm]{Corollary}
\newtheorem{lem}[thm]{Lemma}
\theoremstyle{definition}
\newtheorem{example}[thm]{Example}
\newtheorem{rem}[thm]{Remark}
\newtheorem*{rem*}{Remark}
\newtheorem{definition}[thm]{Definition}
\newtheorem*{conclusion*}{Conclusion}

\newcommand{\mathdcl}[1]{{\ifstrequal{#1}{l}{l}{\textup{#1}}}}
\newcommand{\loc}{\mathdcl{loc}}

\newcommand{\RR}{\mathbb{R}}
\newcommand{\NN}{\mathbb{N}}
\newcommand{\CC}{\mathbb{C}}
\newcommand{\ZZ}{\mathbb{Z}}
\newcommand{\KK}{\mathbb{K}}

\newcommand{\eps}{\varepsilon}
\renewcommand{\phi}{\varphi}

\newcommand{\Linop}{\mathcal{L}}
\newcommand{\conj}[1]{\overline{#1}}
\newcommand{\clos}[1]{\overline{#1}}
\newcommand{\form}[1]{{\mathfrak{#1}}}
\newcommand{\restrict}[2]{\ensuremath{#1\raisebox{-0.4ex}{$|$}\strut_{#2}}}

\renewcommand{\Re}{\operatorname{Re}}

\newcommand{\dx}[1][x]{\,\mathrm{d}#1}

\DeclarePairedDelimiter\norm{\lVert}{\rVert}
\DeclarePairedDelimiter\abs{\lvert}{\rvert}

\makeatletter
\let\old@norm\norm
\@ifpackageloaded{xparse}{
\DeclareDocumentCommand{\smartnorm}{soG{\cdot}}{%
  \IfBooleanTF{#1}{\old@norm{#3}}%
  {%
    \IfValueTF{#2}{\old@norm[#2]{#3}}{\old@norm*{#3}}%
  }%
}
\def\norm{\smartnorm}
}{}
\makeatother

\newcommand{\scalar}[3][auto]{{%
\ifthenelse{\equal{#1}{auto}}{\left(#2\mkern3mu{\mid}\mkern3mu #3\right)}{}
\ifthenelse{\equal{#1}{b}}{\bigl(#2\mkern3mu{\mid}\mkern3mu #3\bigr)}{}
\ifthenelse{\equal{#1}{B}}{\Bigl(#2\mkern3mu{\bigm|}\mkern3mu #3\Bigr)}{}
}}

\newcommand{\emphdef}[1]{\textbf{\boldmath #1\unboldmath}}

\newlist{romanenum}{enumerate}{1}
\setlist[romanenum]{label=\textup{(\roman*)},ref=\textup{(\roman*)},itemsep=0em,topsep=1ex}

\newlist{alenum}{enumerate}{1}
\setlist[alenum]{label=\textup{(\alph*)},ref=\textup{(\alph*)},itemsep=0em,topsep=1ex}

\newlist{parenum}{enumerate}{1}
\setlist[parenum]{label=\textup{\arabic*.},ref=\textup{\arabic*.},wide,noitemsep}

\DeclareMathOperator{\tr}{tr}

\newcommand{\cA}{\mathcal{A}}
\newcommand{\cC}{\mathcal{C}}
\newcommand{\cH}{\mathcal{H}}
\newcommand{\cV}{\mathcal{V}}
\newcommand{\scrS}{\mathscr{S}}

\renewcommand{\MR}{\mathdcl{MR}}
\newcommand{\MRa}{\MR_{\form{a}}}
\renewcommand{\Linop}{\mathscr{L}}
\newcommand{\cMR}{\mathcal{M}\mathcal{R}}

\DeclareDocumentCommand{\PPhi}{o}{%
\mathcal{P}_\Phi\IfValueTF{#1}{(#1)}{}
}
\DeclareDocumentCommand{\Pini}{o}{%
\mathcal{P}\IfValueTF{#1}{(#1)}{}
}

\newcommand{\dembed}{\overset{\mathdcl{d}}{\hookrightarrow}}
\newcommand{\embed}{\hookrightarrow}

\makeatletter
\@namedef{subjclassname@2020}{%
  \textup{2020} Mathematics Subject Classification}
\makeatother

\title{Maximal regularity for generalized boundary conditions in time}

\author[W. Arendt]{\sc Wolfgang Arendt}
\address{Wolfgang Arendt\\Institute of Applied Analysis\\Ulm University\\89069 Ulm\\Germany}
\email{wolfgang.arendt@uni-ulm.de}
\author[M. Sauter]{\sc Manfred Sauter} 
\address{Manfred Sauter\\Faculty of Mathematics and Economics\\Ulm University\\89069 Ulm\\Germany}
\email{manfred.sauter@uni-ulm.de}

\date{6th January 2025} 
\dedicatory{Dedicated to Giuseppe Da Prato}

\keywords{Parabolic evolution equations, non-standard boundary conditions in time, non-autonomous equations, maximal regularity, Lions' problem, non-autonomous forms}
\subjclass[2020]{Primary: 35K90; Secondary: 35B65, 35B30, 47A07, 34G10}

\begin{document}
\begin{abstract}
We consider autonomous and non-autonomous evolution equations on a time interval $[0,\tau]$ in a Banach space $X$ with the non-standard time--boundary condition $u(0)=\Phi u(\tau)$, where $\Phi$ is a linear map on $X$. If $\Phi=0$, this is an initial value problem, whereas $\Phi=I$ corresponds to periodic boundary conditions, and $\Phi=-I$ to antiperiodic boundary conditions.
Our main point is to establish maximal $L^p$-regularity. 
In the non-autonomous case we consider two situations. The first concerns time-dependent operators with a fixed domain.
In the second one we take $X=H$ a Hilbert space and consider evolution equations associated with non-autonomous forms. Of special interest is then maximal regularity in $H$ with a non-standard time--boundary condition.
\end{abstract}

\begingroup
\renewcommand{\MakeUppercase}[1]{#1}
\maketitle
\endgroup

\section{Introduction}

The aim of this paper is to study evolution equations with generalized boundary conditions for the time interval $[0,\tau]$, namely $u(0)=\Phi u(\tau)$,
where $\Phi$ is a linear map on the underlying Banach space.
Classical cases are the initial value problem $u(0)=0$ corresponding to $\Phi=0$, periodic boundary conditions corresponding to $\Phi=I$,
or antiperiodic boundary conditions if $\Phi=-I$.
Our main point is to establish maximal $L^p$-regularity.

The paper is dedicated to the memory of Professor Giuseppe Da Prato, a master and leader in the theory of evolution equations.
The first-named author had the privilege to have him as a charismatic teacher as a third year student at the University of Nice in 1971--1973,
while the famous paper of Da Prato and Grisvard~\cite{DPG75} was written.
This paper, considering the maximal regularity property as the closedness of the sum of two operators, was fundamentally important in the development of the theory of evolution equations.
It established maximal regularity on interpolation spaces and its influence eventually lead to the definite results by Dore and Venni~\cite{DV87} and the characterization of maximal $L^p$-regularity by Weis~\cite{Wei01}, which were the basis of the \textit{max-reg movement} in PDE.

In the paper we start with an operator $A$ enjoying maximal $L^p$-regularity. In Section~\ref{sec:2} we consider the case where $-A$ generates a $C_0$-semigroup, and in Section~\ref{sec:3a} we consider the non-autonomous setting with operators $A(\cdot)$ that share the same domain.
We find the properties of $\Phi$ which lead to weak well-posedness and then maximal $L^p$-regularity for the problem with the generalized boundary condition in time.

The second part of the paper is devoted to non-autonomous evolution equations governed by non-autonomous forms.
This is a purely Hilbert space theory adapted to a Gelfand triple $V\embed H\embed V'$. 
In Section~\ref{sec:3} we introduce the form method first in the autonomous case. Then in Section~\ref{sec:4} we characterize well-posedness in $V'$, a rather weak form, and in Section~\ref{sec:6} well-posedness in $H$ under more restrictive conditions, all for the generalized time--boundary condition. Finally, in Section~\ref{sec:7} we provide an example where one does not have well-posedness in $H$.

To the best of our knowledge such generalized boundary conditions appeared first in the monograph by Showalter~\cite[p.\,112]{Sho97} for selfadjoint problems.
Their origins can be traced back to regularization schemes for certain ill-posed evolution problems with pure initial or final time boundary conditions, cf.~\cite{Sho85,CO94,ACEO98}. The ill-posedness there is due to the lack of continuous dependence on the boundary value.
For example, instead of considering the backward heat equation subject to the final time condition $u(\tau)=y_0$, the equation is considered with a mixed initial and final time boundary condition $\eps u(0)+u(\tau)=y_0$ for small $\eps>0$. (For $y_0=0$ this corresponds to $\Phi x = -\frac{1}{\eps}x$ in our setting.) This regularizes the problem and yields a numerically well-behaved approximation scheme for $\eps\to0+$ for the ill-posed problem.

Equations where the time--boundary condition uses values at different times are also called \textit{nonstandard problems} in the literature. Applications for the considered parabolic, hyperbolic and mixed-type equations include image recovery, elastodynamic models and control theory, see e.g.~\cite{APS04,Qui05,APS05,KP05}.

Throughout this paper we suppose that $\tau>0$ is fixed.

\smallskip
\textbf{Acknowledgment.} We would like to thank Professor Showalter for providing us with information on the origins of these boundary conditions.

\section{Generalized boundary conditions for semigroups}\label{sec:2}

Let $T$ be a $C_0$-semigroup on a real or complex Banach space $X$ with generator $-A$.
We will keep this setting throughout this section.

Given $f\in L^1(0,\tau;X)$ we consider the equation
\begin{equation}\label{eq:inhom-cauchy}
    \dot{u}+Au = f.
\end{equation}
We say that $u$ is a \emphdef{weak solution} of~\eqref{eq:inhom-cauchy} if $u\in C([0,\tau];X)$ and
for all $x'\in D(A')$ we have $\langle x',u\rangle\in W^{1,1}(0,\tau)$ and
\begin{equation}\label{eq:weak-sol}
    \frac{d}{dt}\langle x',u(t)\rangle + \langle A'x', u(t)\rangle = \langle x',f(t)\rangle\quad\text{for a.e.~$t\in(0,\tau)$.}
\end{equation}
We define $T*f\in C([0,\tau];X)$ by
\[
    (T*f)(t) = \int_0^t T(t-s)f(s)\dx[s].
\]

\begin{prop}\label{prop:bal77}\leavevmode
\begin{alenum}
\item\label{en:a-ws}
Let $u\colon[0,\tau]\to X$ be a function. Then $u$ is a weak solution of~\eqref{eq:inhom-cauchy} if and only if there exists an $x\in X$ such that
\[
    u(t) = T(t)x + (T*f)(t)\qquad \text{for all $t\in[0,\tau]$.}
\]
\item
As a consequence, for each $x\in X$ there exists a unique weak solution $u$ of~\eqref{eq:inhom-cauchy} satisfying $u(0)=x$.

\item\label{en:c-ws}
A function $u\in C([0,\tau];X)$ is a weak solution of~\eqref{eq:inhom-cauchy} if and only if $\int_0^t u(s)\dx[s]\in D(A)$ and
\[
    u(t)-u(0) + A\int_0^t u(s)\dx[s] = \int_0^t f(s)\dx[s]
\]
for all $t\in[0,\tau]$.
\end{alenum}
\end{prop}
\begin{proof}
Part~\ref{en:a-ws} is due to~\cite{Bal77}. Part~\ref{en:c-ws} follows from~\cite[Proposition~3.1.16]{ABHN11}.
\end{proof}

If $f\in C([0,\tau];X)$ and $u$ is a weak solution of~\eqref{eq:inhom-cauchy}, then it follows that for $x'\in D(A')$ one has $\langle x',u\rangle\in C^1[0,\tau]$ and~\eqref{eq:weak-sol} holds for all $t\in[0,\tau]$.

Our aim is to consider more general boundary conditions for the time interval other than $u(0)=0$. To that end, let $\Phi\colon X\to X$ be linear.
For $\Phi=I$, the first
equivalence of the following Theorem~\ref{thm:2.2} is proved by Prüss~\cite{Pru93}.

\begin{thm}\label{thm:2.2}
The following statements are equivalent.
\begin{romanenum}
\item\label{en:2.2i} For all $f\in C([0,\tau];X)$ there exists a unique weak solution $u$ of~\eqref{eq:inhom-cauchy} satisfying $u(0)=\Phi u(\tau)$.
\item $I-\Phi T(\tau)\colon X\to X$ is bijective.
\item For all $f\in L^1(0,\tau;X)$ and for all $y_0\in X$ there exists a unique weak solution $u$ of~\eqref{eq:inhom-cauchy} satisfying $u(0)=\Phi u(\tau)+y_0$.
\end{romanenum}
\end{thm}
\begin{proof}
\begin{parenum}
\item[(i)$\Rightarrow$(ii).] Let $x\in X$. Define $f\in C([0,\tau];X)$ by $f(t)=\frac{1}{\tau}T(t)x$. By assumption there exists $y\in X$ such that
\[
    u(t)= T(t)y + (T * f)(t) = T(t)y + \frac{t}{\tau}T(t)x
\]
satisfies $y=u(0)=\Phi u(\tau) = \Phi T(\tau)y + \Phi T(\tau)x$.
Thus $y-\Phi T(\tau)y = \Phi T(\tau)x$.
Hence $(I-\Phi T(\tau))(y+x)=x$.
We have shown that $I-\Phi T(\tau)$ is surjective.
In order to show injectivity, assume that $x-\Phi T(\tau)x=0$. Then $u(t) := T(t)x$ is a weak solution of~\eqref{eq:inhom-cauchy} for $f=0$ satisfying $u(0)=x=\Phi T(\tau)x$.
Hence $u=0$ by the uniqueness assumption in~\ref{en:2.2i}. This implies $x=0$.

\item[(ii)$\Rightarrow$(iii).]
Let $f\in L^1(0,\tau;X)$, $y_0\in X$. Let $u(t)=T(t)x + (T * f)(t)$ be a weak solution of~\eqref{eq:inhom-cauchy} for $x\in X$. Then $u(0)=\Phi u(\tau)+y_0$ if and only if
\begin{equation}\label{eq:2.4}
    x=\Phi T(\tau)x + \Phi((T * f)(\tau)) + y_0.
\end{equation}
Since $I-\Phi T(\tau)$ is invertible, there exists a unique $x\in X$ such that~\ref{eq:2.4} holds.

\item[(iii)$\Rightarrow$(i).] 
This is trivial.\qedhere
\end{parenum}
\end{proof}

If the equivalent conditions of Theorem~\ref{thm:2.2} are satisfied, we say that the problem
\[
    \PPhi\quad \left\{\begin{aligned}
        &\dot{u} + Au = f, \\
        &u(0) = \Phi u(\tau)
    \end{aligned}\right.
\]
is \emphdef{weakly well-posed}.

In that case, for a bounded $\Phi\in\Linop(X)$, the solution depends continuously on the data $f$. In fact the following holds.
\begin{prop}
Let $\Phi\in\Linop(X)$. Assume that $\PPhi$ is weakly well-posed. Let $f,f_n\in L^1(0,\tau;X)$ be such that $\lim_{n\to\infty} f_n=f$ in $L^1(0,\tau;X)$. Let $u_n$ be the solution of $\PPhi[f_n]$. Then $\lim u_n=u$ in $C([0,\tau];X)$, where $u\in C([0,\tau];X)$ is the solution of $\PPhi[f]$.
\end{prop}
\begin{proof}
Denote by $S\colon L^1(0,\tau;X)\to C([0,\tau];X)$ the mapping $Sf=u$, where $u$ is the weak solution of $\PPhi[f]$. Then $S$ is linear. We show that $S$ has a closed graph.
Let $f_n\to f$ in $L^1(0,\tau;X)$. Let $u_n\in C([0,\tau];X)$ be the weak solution of $\PPhi[f_n]$.
Then $u_n(t)=T(t)x_n + (T * f_n)(t)$, where $x_n = (I-\Phi T(\tau))^{-1}\Phi((T*f_n)(\tau))$.
Since $T * f_n\to T*f$ in $C([0,\tau];X)$, it follows that $x_n\to x := (I-\Phi T(\tau))^{-1}\Phi((T * f)(\tau))$. Thus $u_n(t) \to u(t) = T(t)x + (T*f)(t)$ as $n\to\infty$ uniformly in $t\in [0,\tau]$ and $u$ is the weak solution of $\PPhi[f]$.
\end{proof}

Next we want to investigate under which conditions the weak solution is $L^p$-strong.
\begin{definition}
Let $1<p<\infty$. We say that $A$ satisfies \emphdef{maximal $L^p$-regularity}, if for all $f\in L^p(0,\tau;X)$ one has $T*f\in W^{1,p}(0,\tau;X)$. This means that the weak solution $u$ of
\[
    \left\{\begin{aligned}
        &\dot{u}(t)+Au(t)=f(t), \\
        &u(0)=0
    \end{aligned}\right.
\]
is in $\MR_p := W^{1,p}(0,\tau;X)\cap L^p(0,\tau;D(A))$, where $D(A)$ carries the graph norm.
This property does not depend on $1<p<\infty$ or $\tau\in(0,\infty)$.
We write $A\in\cMR(X)$ to say that $A$ satisfies maximal regularity.
\end{definition}

If $A\in\cMR(X)$, then $-A$ generates a holomorphic semigroup by a result of Dore~\cite[Theorem~17.2.15]{HNVW23}, see~\cite{AB10} for a slightly more general result.

The converse is true if $X$ is a Hilbert space: If $T$ is holomorphic, then $A\in\cMR(X)$ (de Simon's Theorem, see~\cite[Corollary~17.3.8]{HNVW23}).
However, it is known that in a large class of Banach spaces, including separable Banach lattices, only Hilbert spaces have the property that the generator of a holomorphic $C_0$-semigroup is automatically in $\cMR(X)$. This is a result of Kalton--Lancien~\cite{KL00}, with a simpler proof due to Fackler~\cite{Fac13}.

On UMD spaces $X$, one has $A\in\cMR(X)$ if and only if $A$ is R-sectorial, see~\cite[Theorem~17.3.1]{HNVW23}.
In our context, it is interesting that a quite simple proof of this result was given in~\cite[Corollary~5.2]{AB02} via the periodic problem, a special case of our subject today.

Now we turn to maximal $L^p$-regularity for the problem $\PPhi$.
Let $1<p<\infty$.
We denote by 
\[
    \tr_p := \{w(0) : w\in \MR_p\}
\]
the \emphdef{trace space of $\MR_p$}. 
If $T$ is holomorphic, then $\tr_p$ coincides with a real interpolation space, namely
\[
    \tr_p = (X,D(A))_{\frac{1}{p'},p},
\]
where $\frac{1}{p}+\frac{1}{p'}=1$.
Moreover,
\[
    \tr_p = \{x\in X : AT(\cdot)x\in L^p(0,\tau; X)\},
\]
see~\cite[Propositions~2.2.2, 1.2.10]{Lun95}.

\begin{definition}
Let $1<p<\infty$ and $f\in L^p(0,\tau;X)$.
A \emphdef{strong $L^p$-solution} of~\eqref{eq:inhom-cauchy} is a function $u\in W^{1,p}(0,\tau;X)$ such that $u(t)\in D(A)$ a.e.\ and $\dot{u}+Au=f$ a.e.
\end{definition}

\begin{prop}
Let $1<p<\infty$ and $f\in L^p(0,\tau;X)$.
\begin{alenum}
\item Each strong $L^p$-solution of~\eqref{eq:inhom-cauchy} is a weak solution.
\item Let $u$ be a weak solution of~\eqref{eq:inhom-cauchy}. Then $u$ is a strong $L^p$-solution of~\eqref{eq:inhom-cauchy} if and only if $u\in W^{1,p}(0,\tau;X)$. In that case $u\in\MR_p$.
\end{alenum}
\end{prop}
\begin{proof}
\begin{alenum}
\item Since $u\in W^{1,p}(0,\tau;X)$, one has $u\in C([0,\tau];X)$ and $u(t)-u(0)=\int_0^t \dot{u}(s)\dx[s]$.
Since $\dot{u}(t)+Au(t)=f(t)$ a.e., it follows that $Au\in L^p(0,\tau;X)$, $\int_0^t u(s)\dx[s]\in D(A)$ and 
\begin{align*}
    A\int_0^tu(s)\dx[s] &= \int_0^t Au(s)\dx[s] = -\int_0^t \dot{u}(s)\dx[s] + \int_0^t f(s)\dx[s] \\ 
    &= -u(t)+u(0) + \int_0^t f(s)\dx[s]
\end{align*}
for all $t\in [0,\tau]$. Thus $u$ is a weak solution by Proposition~\ref{prop:bal77}\,(c).

\item
Let $u\in W^{1,p}(0,\tau;X)$ be a weak solution and let $t\in [0,\tau]$. Then $\int_t^{t+h} u(s)\dx[s]\in D(A)$ and
\[
    \frac{1}{h}(u(t+h)-u(t)) + \frac{1}{h} A\int_t^{t+h}u(s)\dx[s] = \frac{1}{h}\int_t^{t+h}f(s)\dx[s]
\]
if $t+h\in [0,\tau]$. Recall that for almost every $t\in [0,\tau]$ one has
\[
    \lim_{h\to 0}\frac{1}{h}(u(t+h)-u(t))=\dot{u}(t)
\]
and
\[
    \lim_{h\to 0}\frac{1}{h}\int_t^{t+h}f(s)\dx[s] = f(t).
\]
Since $A$ is closed, it follows that $\dot{u}(t)+Au(t)=f(t)$ a.e.\qedhere
\end{alenum}
\end{proof}

The existence of strong $L^p$-solutions, without any time--boundary conditions being imposed, already implies that $A\in\cMR(X)$. In fact, the following holds.

\begin{prop}\label{prop:2.5}
Assume that $T$ is holomorphic. Let $1<p<\infty$. 
Assume that for all $f\in L^p(0,\tau;X)$ there exists a $u\in\MR_p$ such that $\dot{u}+Au=f$.
Then $A\in\cMR(X)$.
\end{prop}
\begin{proof}
Let $f\in L^p(0,\tau;X)$ and let $u\in\MR_p$ be such that $\dot{u}+Au=f$.
Then $x:= u(0)\in\tr_p$. 
Let $w=T*f$. Then $w\in C([0,\tau];X)$ is a weak solution of $\dot{w}+Aw=f$ and $w(0)=0$.
Let $v=u-w$. Then $v\in C([0,\tau];X)$ is a weak solution of $\dot{v}+Av=0$ and $v(0)=x$.
Thus $v(t)=T(t)x$. Since $x\in\tr_p$, one has $v\in\MR_p$.
It follows that $T*f=u-v\in\MR_p$,
which proves the claim.
\end{proof}

Next we investigate strong $L^p$-solutions for generalized boundary conditions in time.
\begin{definition}
Let $1<p<\infty$ and let $\Phi\colon\tr_p\to X$ be linear. 
We say that problem $\PPhi$ satisfies \emphdef{maximal $L^p$-regularity}, if for all $f\in L^p(0,\tau;X)$ there exists a unique strong $L^p$-solution $u$ of~\eqref{eq:inhom-cauchy} satisfying $u(0)=\Phi u(\tau)$.
As a consequence, all the three terms $\dot{u}$, $Au$, $f$ are in $L^p(0,\tau;X)$.
\end{definition}

We already know that $A\in\cMR(X)$ whenever $\PPhi$ satisfies maximal $L^p$-regularity for some $1<p<\infty$ and the semigroup $T$ is holomorphic.
We state the following further consequences of maximal $L^p$-regularity for $\PPhi$.

\begin{prop}\label{prop:2.9new}
Assume that $T$ is holomorphic.
Let $1<p<\infty$ and let $\Phi\colon\tr_p\to X$ be linear.
Assume that $\PPhi$ satisfies maximal $L^p$-regularity.
Then the following holds.

\begin{alenum}
\item For all $y\in\tr_p$ there exists a unique $x\in\tr_p$ such that $(I-\Phi T(\tau))x=y$.
\item The map $I-\Phi T(\tau)\colon X\to X$ is surjective.
\item If in addition
\begin{equation}\label{eq:2.7}
    x-\Phi T(\tau) x = 0\quad\text{implies $x=0$, for every $x\in X$,}
\end{equation}
then $\Phi\tr_p\subset\tr_p$.
\end{alenum}
\end{prop}
\begin{proof}
\begin{parenum}
\item Let $x\in\tr_p$ be such that $\Phi T(\tau)x=x$.
Then $u=T(\cdot)x\in\MR_p$,
$\dot{u}+Au=0$ and $u(0)=\Phi u(\tau)$.
Thus $u=0$ and $x=0$. We have shown that $I-\Phi T(\tau)\colon\tr_p\to X$ is injective.

\item Let $y\in X$, $f(t)=\frac{1}{\tau} T(t)y$. Then $f\in L^p(0,\tau;X)$ and $(T*f)(\tau)=T(\tau)y$. Since $A\in\cMR(X)$, $T*f\in\MR_p$. By assumption there exists a unique $u\in\MR_p$ such that $\dot{u}+Au=f$, $u(0)=\Phi u(\tau)$.
Then $v=u-T*f\in\MR_p$ and $\dot{v}+Av=0$. Thus $v(t)=T(t)x$, where $x=u(0)\in\tr_p$.
Since $\Phi u(\tau)=x$, it follows that
\[
    \Phi T(\tau)x + \Phi T(\tau)y = \Phi u(\tau) = u(0) = x.
\]
Hence
\[
    x-\Phi T(\tau)x = \Phi T(\tau)y.
\]
This is equivalent to
\begin{equation}\label{eq:2.9}
    (I-\Phi T(\tau))(x+y) = y.
\end{equation}
We have shown that for all $y\in X$ there exists a unique $x\in\tr_p$ such that~\eqref{eq:2.9} holds.
Consequently, for all $y\in X$ there exists a $z\in X$ such that
\begin{equation}\label{eq:2.10}
    (I-\Phi T(\tau))z = y.
\end{equation}
Moreover, if $y\in\tr_p$, then $x+y\in\tr_p$. Thus there exists a unique $z\in\tr_p$ such that~\eqref{eq:2.10} holds. Thus~(a) and~(b) are proved.

\item\label{en:step3}
Assume in addition~\eqref{eq:2.7}.
Then for $y,z\in X$ satisfying~\eqref{eq:2.10} one has $y\in\tr_p$ if and only if $z\in\tr_p$.
In fact, assume that $z\in\tr_p$. By~\eqref{eq:2.9} there exists $x\in\tr_p$ such that
$(I-\Phi T(\tau))(x+y)=y$. Then~\eqref{eq:2.7} implies that $z=x+y$. Hence $y=z-x\in\tr_p$.

\item
Let $y\in\tr_p$. We show that $\Phi y\in\tr_p$. Since $y\in\tr_p$, there exists $w\in\MR_p$ such that $w(\tau)=y$. By assumption, there exists $u\in\MR_p$ such that $\dot{u}+Au=f :=\dot{w}+Aw$ and $u(0)=\Phi u(\tau)$. Then $v := u-w\in\MR_p$ and $\dot{v}+Av=0$. 
Consequently, there exists $x\in\tr_p$ such that $v(t)= T(t)x$ for all $t\in[0,\tau]$.
Since $u(0)=\Phi u(\tau)$, it follows that $x-\Phi T(\tau)x=\Phi y-w(0)$.

Now~\ref{en:step3} implies that $\Phi y-w(0)\in\tr_p$. Since $w(0)\in\tr_p$, this implies $\Phi y\in\tr_p$.\qedhere
\end{parenum}
\end{proof}

\begin{prop}\label{prop:2.10new}
Assume that $T$ is holomorphic.
Let $1<p<\infty$ and let $\Phi\colon\tr_p\to\tr_p$ be linear.
Then the following statements are equivalent.
\begin{romanenum}
\item $A\in\cMR(X)$ and $I-\Phi T(\tau)\colon \tr_p\to\tr_p$ is bijective. 
\item $\PPhi$ satisfies maximal $L^p$-regularity.
\end{romanenum}
\end{prop}
\begin{proof}
\begin{parenum}
\item[(i)$\Rightarrow$(ii).]
Let $f\in L^p(0,\tau;X)$. Then $(T*f)(\tau)\in\tr_p$ since $T*f\in\MR_p$.
By~(i) there exists a unique $x\in\tr_p$ such that
\[
    x-\Phi T(\tau)x = \Phi((T*f)(\tau)).
\]
Let $u(t) = T(t)x + (T*f)(t)$. Then $u\in\MR_p$, $\dot{u}+Au=f$ and $\Phi u(\tau) = \Phi T(\tau)x + \Phi((T*f)(\tau))=x=u(0)$. This shows existence in~(ii).
In order to show uniqueness, let $u\in\MR_p$ be such that $\dot{u}+Au=0$, $u(0)=\Phi u(\tau)$.
Then there exists $x\in\tr_p$ such that $u(t)=T(t)x$. Hence $\Phi T(\tau)x=x$.
Now~(i) implies $x=0$.

\item[(ii)$\Rightarrow$(i).]
 Follows from Proposition~\ref{prop:2.9new}.\qedhere
\end{parenum}
\end{proof}

If we assume weak well-posedness, then we obtain the following characterization.
\begin{thm}\label{thm:2.10new}
Assume that $T$ is holomorphic.
Let $1<p<\infty$.
Let $\Phi\colon X\to X$ be linear such that $I-\Phi T(\tau)\colon X\to X$ is bijective.
Then the following statements are equivalent.
\begin{romanenum}
\item $\PPhi$ satisfies maximal $L^p$-regularity.
\item $A\in\cMR(X)$, $\Phi \tr_p\subset\tr_p$ and $(I-\Phi T(\tau))^{-1}\tr_p\subset\tr_p$.
\end{romanenum}
\end{thm}
\begin{proof}
(i)$\Rightarrow$(ii) follows from Proposition~\ref{prop:2.9new}, and  (ii)$\Rightarrow$(i) follows from Proposition~\ref{prop:2.10new}.
\end{proof}

By specializing to periodic boundary conditions, we obtain the following.
\begin{cor}\label{cor:2.8}
Assume that the semigroup $T$ is holomorphic and that $I-T(\tau)\in\Linop(X)$ is invertible.
Let $1<p<\infty$. Then the following statements are equivalent.
\begin{romanenum}
\item For all $f\in L^p(0,\tau;X)$ there exists a unique $u\in W^{1,p}(0,\tau;X)\cap L^p(0,\tau;D(A))$ such that
\[
    \left\{\begin{aligned}
        &\dot{u}+Au = f, \\
        &u(0) = u(\tau).
    \end{aligned}\right.
\]
\item $A\in\cMR(X)$.
\end{romanenum}
\end{cor}
\begin{proof}
\begin{parenum}
\item[(i)$\Rightarrow$(ii).] Proposition~\ref{prop:2.5}.
\item[(ii)$\Rightarrow$(i).]
Observe that $T(\tau)X\subset D(A)\subset\tr_p$.
Let $y\in\tr_p$, $x=(I-T(\tau))^{-1}y$.
Then $x=T(\tau)x+y\in\tr_p$.
Thus~(ii) of Theorem~\ref{thm:2.10new} is satisfied with $\Phi=I$.\qedhere
\end{parenum}
\end{proof}

In the setting of Corollary~\ref{cor:2.8} where $\Phi=I$, the question of whether $\PPhi$ satisfies $L^p$-maximal regularity is independent of $p\in(1,\infty)$.
However, there exist $\Phi$ for which maximal $L^p$-regularity of $\PPhi$ depends on $p\in(1,\infty)$.

\begin{example}
Assume that $A\in\cMR(X)$.
Let $1<p<q<\infty$. Then $\tr_q\subset\tr_p$. Choose an example where $\tr_q\neq \tr_p$. (This can be arranged using~\cite[Example~1.10]{Lun09}, for example.)
Take $x_0\in\tr_p\setminus\tr_q$ and $x'\in X'$ such that $\langle x',x_0\rangle\ne 0$ and $\norm{T(\tau)}_{\Linop(X)}\norm{x_0}_X\norm{x'}_{X'}<1$.
Let $\Phi\in\Linop(X)$ be the rank-$1$ operator given by $\Phi x = \langle x',x\rangle x_0$.
Then by Theorem~\ref{thm:2.10new} the problem $\PPhi$ is maximal $L^p$-regular, but not maximal $L^q$-regular.
\end{example}

Thus we have obtained a quite complete characterization of when $\PPhi$ satisfies maximal $L^p$-regularity.
However, we were assuming throughout that $-A$ is the generator of a $C_0$-semigroup.
This condition is necessary if $\Phi=0$ (see~\cite[Theorem~4.1]{AB10} for a precise formulation), but not if $\Phi=I$.
In fact, the following result holds on a Hilbert space $X=H$, see~\cite[Theorem~3.1]{AB10}.
\begin{thm}
Let $D,H$ be Hilbert spaces such that $D\embed H$.
Let $1<p<\infty$ and let $A\in\Linop(D,H)$. Then the following statements are equivalent.
\begin{romanenum}
\item For all $f\in L^p(0,2\pi;H)$ there exists a unique $u\in W^{1,p}(0,2\pi;H)\cap L^p(0,2\pi;D)$ such that
\[
    \left\{\begin{aligned}
        &\dot{u}+Au = f, \\
        &u(0) = u(2\pi).
    \end{aligned}\right.
\]
\item For every $k\in\ZZ$ the operator $ik-A$ is invertible and
\[
    \sup_{k\in\ZZ}\norm{(ik-A)^{-1}}_{\Linop(H,D)}<\infty.
\]
\end{romanenum}
\end{thm}
Condition~(ii) implies that $A$ is closed, seen as an unbounded operator in $H$, that $D$ is dense in $H$ and that there exists a bisector
\[
    B = \{re^{i\alpha}:r\ge r_0,\ \abs{\alpha}\in(\tfrac{\pi}{2}-\theta,\tfrac{\pi}{2}+\theta)\}
\]
with $r_0>0$ and $\theta\in(0,\frac{\pi}{2})$ such that $\lambda-A$ is invertible whenever $\lambda\in B$ and
\begin{equation}\label{eq:2.11}
    \sup_{\lambda\in B}\norm{\lambda(\lambda-A)^{-1}}_{\Linop(H)}<\infty.
\end{equation}
This condition does not imply that $-A$ generates a $C_0$-semigroup.
But if $-A$ generates a $C_0$-semigroup $T$, then~\eqref{eq:2.11} implies that $T$ is holomorphic.

\section{Non-autonomous evolution equations: constant domain}\label{sec:3a}

Let $X,D$ be Banach spaces such that $D\dembed X$ (i.e., $D$ is continuously and densely embedded in $X$).
For $1<p<\infty$ we consider
\[
    \MR_p := W^{1,p}(0,\tau;X)\cap L^p(0,\tau;D)
\]
with the corresponding trace space $\tr_p:=\{u(0) : u\in\MR_p\}$.
Then $\MR_p\subset C([0,\tau];\tr_p)$. Moreover, $\tr_p=(X,D)_{\frac{1}{p'},p}$.

In this section we consider operators $A\in\Linop(D,X)$ such that $A\in\cMR(X)$ if we set $D(A) := D\embed X$. This means $A$ with domain $D$ in $X$ satisfies maximal $L^p$-regularity for one, and hence all, $p\in(1,\infty)$.
Passing to the non-autonomous setting we now consider a continuous function $A\colon [0,\tau]\to\Linop(D,X)$ such that $A(t)\in\cMR(X)$ for all $t\in [0,\tau]$.
We keep this assumption throughout this section and fix $1<p<\infty$.
We recall the following result, cf.~\cite[Theorem~2.7]{ACFP07}.

\begin{thm}\label{thm:3a.1}
For all $f\in L^p(0,\tau;X)$ and all $x\in\tr_p$ there exists a unique $u\in\MR_p$ such that
\[
    \left\{\begin{aligned}
        &\dot{u}+A(\cdot)u = f,\\
        &u(0)=x.
    \end{aligned}\right.
\]
Here $\dot{u}+A(\cdot)u=f$ is an identity in $L^p(0,\tau;X)$; i.e., $\dot{u}(t)+A(t)u(t)=f(t)$ a.e.
\end{thm}

\begin{cor}
There exists a unique operator $S(\tau)\in\Linop(\tr_p)$ such that for $x\in\tr_p$ one has $S(\tau)x=v(\tau)$, where $v\in\MR_p$ satisfies
\[
    \left\{\begin{aligned}
        &\dot{v}+A(\cdot)v=0,\\
        &v(0)=x.
    \end{aligned}\right.
\]
\end{cor}
Now we can prove similarly as in Section~\ref{sec:2} the following well-posedness result for generalized time--boundary conditions.
\begin{thm}
Let $\Phi\colon\tr_p\to\tr_p$ be linear. The following statements are equivalent.
\begin{romanenum}
\item $I-\Phi S(\tau)\colon\tr_p\to\tr_p$ is bijective.
\item For all $f\in L^p(0,\tau;X)$ there exists a unique $u\in\MR_p$ such that
\[
    \left\{\begin{aligned}
        &\dot{u}+A(\cdot)u = f,\\
        &u(0)=\Phi u(\tau).
    \end{aligned}\right.
\]
\end{romanenum}
\end{thm}
\begin{proof}
\begin{parenum}
\item Uniqueness in~(ii) is equivalent to injectivity of $I-\Phi S(\tau)$.
In fact, assume uniqueness in~(ii). Let $x\in\tr_p$ be such that $x-\Phi S(\tau)x=0$.
There exists $v\in\MR_p$ such that $\dot{v}+A(\cdot)v=0$ and $v(0)=x$.
Thus $v$ is a solution of with generalized time--boundary conditions for $f=0$.
Uniqueness implies that $v=0$ and therefore $x=0$.
Conversely, assume that $I-\Phi S(\tau)$ is injective. Let $v\in\MR_p$ be such that $\dot{v} + A(\cdot)v=0$, $v(0)=\Phi v(\tau)$.
Then $x:= v(0)\in\tr_p$ and $v(\tau)=S(\tau)x$. Thus $x-\Phi S(\tau)x=0$.
Hence $x=0$. It follows from Theorem~\ref{thm:3a.1} that $v=0$.

\item (i)$\Rightarrow$(ii).
Let $f\in L^p(0,\tau;X)$. By Theorem~\ref{thm:3a.1} there exists a $w\in\MR_p$ such that
$\dot{w}+A(\cdot)w=f$, $w(0)=0$. Then $w(\tau)\in\tr_p$. By~(i) there exists an $x\in\tr_p$ such that $x-\Phi S(\tau)x=\Phi w(\tau)$.
There exists a $v\in\MR_p$ such that $\dot{v}+A(\cdot)v=0$, $v(0)=x$. 
Then $u:=v+w\in\MR_p$, $\dot{u}+A(\cdot)u=f$, $u(0)=x$ and
\[
    \Phi u(\tau) = \Phi v(\tau) + \Phi w(\tau) = \Phi S(\tau)x + (x-\Phi S(\tau)x)=x.
\]

\item (ii)$\Rightarrow$(i).
We have to show that $I-\Phi S(\tau)$ is surjective. Let $y\in\tr_p$. Then $S(\tau)y\in\tr_p$.
Thus there exists $w\in\MR_p$ such that $w(\tau)=S(\tau)y$.
Replacing $w$ by $\phi w$ where $\phi\in C^1[0,\tau]$ with $\phi(0)=0$ and $\phi(\tau)=1$, we may assume that $w(0)=0$.
Then $\dot{w}+A(\cdot)w\in L^p(0,\tau;X)$.
By~(ii) there exists $u\in\MR_p$ such that $\dot{u}+A(\cdot)u=\dot{w}+A(\cdot)w$ and $u(0)=\Phi u(\tau)$. Let $v=u-w$. Then $v(0)=u(0)=:x$ and
$S(\tau)x = v(\tau)= u(\tau)-w(\tau)=u(\tau)-S(\tau)y$ as well as
$x=\Phi u(\tau) = \Phi S(\tau)x + \Phi S(\tau)y$.
Hence $x-\Phi S(\tau)x = \Phi S(\tau)y$. Consequently, $(I-\Phi S(\tau))(x+y)=y$.\qedhere
\end{parenum}
\end{proof}

In general it is difficult to verify~(i), even in the case of periodic boundary conditions (i.e., even if $\Phi=I$). In fact, it might be difficult to determine $\norm{S(\tau)}_{\Linop(\tr_p)}$ or the spectrum of $S(\tau)$ as an operator in $\Linop(\tr_p)$.
Frequently, more information is available for the corresponding operator in $X$.
Recall that an operator $B$ in $X$ is called \emphdef{accretive} if
\[
    \norm{x+tBx}_X\ge \norm{x}_X\quad\text{for all $x\in D(B)$ and $t>0$.}
\]
If $X=H$ is a Hilbert space, this is equivalent to 
\[
    \Re\scalar{Bx}{x}_H\ge 0\quad\text{for all $x\in D(B)$.}
\]
We recall the following well-posedness result from~\cite[Corollary~3.4]{ACFP07}.
Let $\MR_{p,\loc} := C([0,\tau];X)\cap W^{1,p}_\loc((0,\tau];X)\cap L^p_\loc((0,\tau];D)$.
\begin{prop}\label{prop:3a.4}
Assume that $A(t)$ is accretive for all $t\in [0,\tau]$.
Then for all $x\in X$ there is a unique 
\begin{equation}\label{eq:3a.1}
    \begin{cases}
        \text{$u\in\MR_{p,\loc}$ such that}\\
        \begin{aligned}
        &\dot{u}+A(\cdot)u =0,\\
        &u(0)=x.
        \end{aligned}
    \end{cases}
\end{equation}
Moreover, $\norm{u(t)}_X\le\norm{x}_X$ for all $t\in [0,\tau]$. 
Finally, the function $\tilde{u}$ given by $\tilde{u}(t) := t u(t)$ is in $\MR_p$ and
\begin{equation}\label{eq:3a.2}
    \norm{\tilde{u}}_{W^{1,p}(0,\tau;X)} + \norm{\tilde{u}}_{L^p(0,\tau;D)}\le M\norm{x}_X,
\end{equation}
where $M\ge 0$ is a constant independent of $x$.
\end{prop}

We define the operator $T(\tau)\in\Linop(X)$ by $T(\tau)x=u(\tau)$, where $u$ satisfies~\eqref{eq:3a.1}, and call $T(\tau)$ the \emphdef{solution operator for $A(\cdot)$}. Thus $\norm{T(\tau)}_{\Linop(X)}\le 1$.
By Theorem~\ref{thm:3a.1}, $T(\tau)\tr_p\subset\tr_p$ and $S(\tau)=\restrict{T(\tau)}{\tr_p}$.
Now we obtain a result for generalized time--boundary conditions where the linear map $\Phi\colon X\to X$ might not leave invariant the trace space $\tr_p$.
We will use the following straightforward rescaling property.
\begin{lem}\label{lem:3a.5}
Let $f\in L^p(0,\tau;X)$ and $u\in\MR_{p,\loc}$ be such that $\dot{u}+A(\cdot)u=f$.
Let $\alpha\in\RR$ and $v(t)=e^{\alpha t} u(t)$.
Then $v\in\MR_{p,\loc}$ and $\dot{v}(t)+(A(t)-\alpha)v(t)=e^{t\alpha} f(t)$ a.e.
The solution operator $T_\alpha(\tau)$ for $A(\cdot)-\alpha$ is $T_\alpha(\tau)=e^{\alpha\tau}T(\tau)$.
Moreover, if $\Phi\colon X\to X$ is linear, then $\Phi u(\tau)=u(0)$ if and only if
$e^{-\alpha\tau}\Phi v(\tau)=v(0)$. Thus $\Phi$ is replaced by $\Phi_\alpha := e^{-\alpha\tau}\Phi$ if we replace $A(\cdot)$ by $A(\cdot)-\alpha$.
\end{lem}
Thus, if $A(t)-\alpha$ is accretive for all $t\in [0,\tau]$ and $T_\alpha(\tau)\colon X\to X$ is the solution operator of $A(\cdot)-\alpha$, then $\norm{T_\alpha(\tau)}_{\Linop(X)}\le 1$ and therefore $\norm{T(\tau)}_{\Linop(X)}\le e^{-\alpha\tau}$.

\begin{thm}\label{thm:3a.6}
Let $\Phi\colon X\to X$ be linear. Assume that there is an $\alpha\in\RR$ such that $A(t)-\alpha$ is accretive for all $t\in [0,\tau]$.
Assume also that $I-\Phi T(\tau)\colon X\to X$ is bijective.
Then for all $f\in L^p(0,\tau;X)$ there exists a unique
\begin{equation}\label{eq:3a.3}
    \begin{cases}
    \text{$u\in\MR_{p,\loc}$ such that}\\
    \begin{aligned}
        &\dot{u}+A(\cdot)u = f,\\    
        &u(0)=\Phi u(\tau). 
    \end{aligned}
    \end{cases}
\end{equation}
\end{thm}
\begin{proof}
Lemma~\ref{lem:3a.5} allows us to assume that $\alpha=0$.
\begin{parenum}
\item Let $u\in\MR_{p,\loc}$ be such that $\dot{u}+A(\cdot)u=0$. Let $x:=u(0)$. Then $u(0)=\Phi u(\tau)$ if and only if $x-\Phi T(\tau)x=0$. In that case, by Theorem~\ref{thm:3a.1}, $x=0$ if and only if $u=0$.
We have shown that uniqueness in~\eqref{eq:3a.3} is equivalent to $I-\Phi T(\tau)$ being injective.

\item Let $f\in L^p(0,\tau;X)$. By Theorem~\ref{thm:3a.1} there exists a unique $w\in\MR_p$ such that $\dot{w}+A(\cdot)w=f$, $w(0)=0$. By assumption, there exists $x\in X$ such that $x-\Phi T(\tau)x=\Phi w(\tau)$.
By Proposition~\ref{prop:3a.4} there exists $v\in\MR_{p,\loc}$ such that $\dot{v}+A(\cdot)v=0$, $v(0)=x$. Then $v(\tau)=T(\tau)x$. Let $u=w+v$. Then $u$ is a solution of~\eqref{eq:3a.3}.\qedhere
\end{parenum}
\end{proof}

\begin{cor}\label{cor:3a.7}
Let $\alpha\in\RR$ and assume that $A(t)-\alpha$ is accretive for all $t\in [0,\tau]$.
Let $\Phi\in\Linop(X)$ be such that $\norm{\Phi}_{\Linop(X)}<e^{\alpha\tau}$.
Then for all $f\in L^p(0,\tau;X)$ there exists a unique $u\in\MR_{p,\loc}$ such that
\[
    \left\{\begin{aligned}
        &\dot{u} + A(\cdot)u = f,\\
        &u(0) = \Phi u(\tau).
    \end{aligned}\right.
\]
\end{cor}
\begin{proof}
Since $\norm{T(\tau)}_{\Linop(H)}\le e^{-\alpha\tau}$, one has $\norm{\Phi T(\tau)}_{\Linop(H)}<1$ and $I-\Phi T(\tau)$ is invertible.
\end{proof}

We now make an additional compactness assumption, which is verified in many examples.
First we establish the following version of the Aubin--Lions lemma (see~\cite[Exercise~19.3]{AVV25} for the case of Hilbert spaces).
In the original version of Aubin and of Lions, an embedding in $L^p(a,b;X)$ is considered.

\begin{lem}\label{lem:3a.8}
Assume that the embedding $D\embed X$ is compact.
Let $-\infty<a<b<\infty$. Then the embedding
\[
    \MR_p(a,b):= W^{1,p}(a,b;X)\cap L^p(a,b;D)\embed C([a,b];X)
\]
is compact for any $1<p<\infty$.
\begin{proof}
We choose on $\MR_p(a,b)$ the norm 
\[
    \norm{v}_{\MR_p} := \norm{v}_{W^{1,p}(a,b;X)} + \norm{v}_{L^p(a,b;D)}.
\]
\begin{parenum}
\item Let $v\in W^{1,p}(a,b;X)$ be such that $\norm{v}_{\MR_p}\le 1$. Then
\[
    \norm{v(t)-v(t_0)}_X = \norm{\int_t^{t_0} \dot{v}(s)\dx[s]}_X \le \abs{t-t_0}^{\frac{1}{p'}}\Bigl(\int_a^b \norm{\dot{v}(s)}_X^p\dx[s]\Bigr)^{\frac{1}{p}}\le \abs{t-t_0}^{\frac{1}{p'}}.
\]
This shows that the set
\[
    K := \{v\in\MR_p(a,b): \norm{v}_{\MR_p}\le 1\}\subset C([a,b];X)
\]
is equicontinuous at any $t_0\in[a,b]$.

\item Let $t\in(a,b]$. We show that the set $K(t) := \{v(t) : v\in K\}$ is precompact in $X$.
Let $\eps>0$. Choose $h>0$ so small that $t-h\ge a$ and $(\frac{1}{p'+1} h)^{\frac{1}{p'}} <\eps$. 
If $v\in K$, then
\[
    v(t) = \frac{1}{h}\int_{t-h}^t \frac{\mathdcl{d}}{\mathdcl{d}s}\bigl((s-(t-h))v(s)\bigr)\dx[s] = \frac{1}{h}\int_{t-h}^t\bigl(s-(t-h)\bigr)v'(s)\dx[s] + \frac{1}{h}\int_{t-h}^t v(s)\dx[s].
\]
Then
\[
    \norm{\frac{1}{h}\int_{t-h}^t\bigl(s-(t-h)\bigr)v'(s)\dx[s]}\le \frac{1}{h}\Bigl(\int_{t-h}^t\bigl(s-(t-h)\bigr)^{p'}\dx[s]\Bigr)^{\frac{1}{p'}}\norm{v'}_{L^p(a,b;X)}\le\eps.
\]
Since $\norm{v}_{L^p(a,b;D)}\le 1$, the set $\{\frac{1}{h}\int_{t-h}^t v(s)\dx[s] : v\in K\}$ is bounded in $D$, and so it is precompact in $X$. It follows that the set $K(t)$ can be covered by finitely many $\eps$-balls in $X$. The proof for $t=a$ is similar.

\item Now it follows from the Arzelà--Ascoli theorem that the unit ball of $\MR_p(a,b)$ is precompact in $C([a,b];X)$.\qedhere
\end{parenum}
\end{proof}
\end{lem}

\begin{rem}
Even though $\MR_p\subset C([0,\tau];\tr_p)$, this latter embedding is not compact unless $\tr_p$ is finite-dimensional.
In fact, by Theorem~\ref{thm:3a.1} the closed subspace $\{u\in\MR_p : \dot{u} + A(\cdot)u=0\}$
of $\MR_p$ is isomorphic to $\tr_p$.
\end{rem}

\begin{prop}\label{prop:3a.10}
Assume that the embedding $D\embed X$ is compact. Then the operator $T(\tau)\colon X\to X$ is compact.
\end{prop}
\begin{proof}
Let $0<\eps<\tau$. By Lemma~\ref{lem:3a.8} the embedding $\MR_p(\eps,\tau)\embed C([\eps,\tau];X)$ is compact.
Let $x_n\in X$ be such that $\norm{x_n}_X\le 1$. Let $v_n\in\MR_{p,\loc}$ be such that $\dot{v}_n+A(\cdot)v_n = 0$, $v(0)=x_n$. Then by~\eqref{eq:3a.2} one has $\norm{v_n}_{\MR_p(\eps,\tau)}\le\frac{M}{\eps}$ for all $n\in\NN$.
By Lemma~\ref{lem:3a.8} there exists a subsequence such that $v_{n_k}$ converges in $C([\eps,\tau];X)$ as $k\to\infty$. In particular, $S(\tau)x_{n_k} = v_{n_k}(\tau)$ converges as $k\to\infty$.
\end{proof}

Now we can use the Fredholm alternative and obtain the following characterization.
\begin{thm}
Assume that the embedding $D\embed V$ is compact.
Furthermore, assume that there exists $\alpha\in\RR$ such that $A(t)-\alpha$ is accretive for all $t\in [0,\tau]$. Let $\Phi\in\Linop(X)$.
Then the following statements are equivalent.
\begin{romanenum}
\item $\Phi T(\tau)x=x$ implies $x=0$, for all $x\in X$.
\item For all $f\in L^p(0,\tau;X)$ there exists a unique $u\in\MR_{p,\loc}$ such that
\[
    \left\{\begin{aligned}
        &\dot{u}+A(\cdot)u = f,\\
        &u(0)=\Phi u(\tau).
    \end{aligned}\right.
\]
\end{romanenum}
\end{thm}
\begin{proof}
We may assume that $\alpha=0$, see Lemma~\ref{lem:3a.5}.
\begin{parenum}
\item[(i)$\Rightarrow$(ii).] 
Since $\Phi T(\tau)$ is compact by Proposition~\ref{prop:3a.10}, it follows that $I-\Phi T(\tau)\in\Linop(X)$ is invertible. So~(ii) follows from Theorem~\ref{thm:3a.6}.

\item[(ii)$\Rightarrow$(i).]
Let $x\in X$ be such that $\Phi T(\tau)x=x$. Let $v\in\MR_{p,\loc}$ be such that $\dot{v}+A(\cdot)v=0$ and $v(0)=x$. Then $T(\tau)x=v(\tau)$ and so $\Phi v(\tau)=v(0)$.
Uniqueness in~(ii) implies that $v=0$. Hence $x=0$.\qedhere
\end{parenum}
\end{proof}

\begin{example}
Let $\KK=\RR$ and $\Omega\subset\RR^d$ be open and bounded of class $C^2$.
Assume that $a_{lk}=a_{kl}\in C([0,\tau]\times\clos{\Omega})$ for all $k,l=1,\dots,d$ and suppose that there exists a $\beta>0$ such that
\[
    \sum_{k,l=1}^d a_{kl}(t,x)\xi_l\xi_k\ge \beta\abs{\xi}^2
\]
for all $\xi\in\RR^d$, $t\in [0,\tau]$, $x\in\clos{\Omega}$.
Moreover, assume that $a_{kl}(t,\cdot)\in W^{1,\infty}(\Omega)$ for every $t\in [0,\tau]$ and $\partial_k a_{kl}\in L^\infty((0,\tau)\times\Omega)$ for all $k,l=1,\dots,d$.
Let $1<q<\infty$, $X=L^q(\Omega)$, $D=W^{2,q}(\Omega)\cap W^{1,q}_0(\Omega)$.
Define $A(t)\in\Linop(D,X)$ by 
\[
    A(t)v = -\sum_{k,l=1}^d \partial_k\bigl(a_{kl}(t,\cdot)\partial_l v\bigr).
\]
Then $A\colon [0,\tau]\to\Linop(D,X)$ is continuous (see~\cite[Section~4]{ACFP07}).
Denote by $\lambda_1>0$ the first eigenvalue of $-\Delta_\mathdcl{D}$ (where $\Delta_\mathdcl{D}$ is the Dirichlet Laplacian).
Let
\[
    \alpha_q = \begin{cases}
        \beta\lambda_1\frac{2}{q} & \text{if $q\ge 2$, and} \\
        \beta\lambda_1\frac{2}{q'} & \text{if $q\le 2$.}
    \end{cases}
\]
Then $A(t)-\alpha_q$ is accretive for all $t\in [0,\tau]$.
\begin{proof}
By the Poincaré inequality,
\[
    \lambda_1\norm{v}_{L^2(\Omega)}^2\le \int_\Omega \abs{\nabla v}^2
\]
for all $v\in H^2(\Omega)\cap H^1_0(\Omega)$.

Fix $t\in [0,\tau]$ and denote by $(T_q(s))_{s\ge 0}$ the $C_0$-semigroup generated by $-A(t)$ on $X=L^q(\Omega)$.
If $q=2$, then
\[
    \scalar{A(t)v}{v}_{L^2(\Omega)} \ge\beta\int_\Omega\abs{\nabla v}^2\ge \lambda_1\beta\norm{v}_{L^2(\Omega)}^2.
\]
So $A(t)-\lambda_1\beta$ is accretive for the choice $q=2$.

By the Beurling--Deny criterion there exist contractive operators $T_1(s)\in\Linop(L^1(\Omega))$ and $T_\infty(s)\in\Linop(L^\infty(\Omega))$ which are consistent with $T_2(s)$.

Now the Riesz--Thorin theorem implies that for $q\in(1,\infty)$ one has $\norm{T_q(s)}\le e^{-\alpha_q s}$ for all $s\ge 0$.
Thus $A(t)-\alpha_q$ is accretive in $X=L^q(\Omega)$.
\end{proof}

Let $\Phi\in\Linop(L^q(\Omega))$ be such that
$\norm{\Phi}_{\Linop(L^q(\Omega))}<e^{\alpha_q\tau}$.
Thus we obtain from Corollary~\ref{cor:3a.7} the following result.
Let $1<p<\infty$.
For all $f\in L^p(0,\tau;L^q(\Omega))$ there exists a unique $u\in W^{1,p}(0,\tau;L^q(\Omega))\cap L^p(0,\tau;W^{2,q}(\Omega))$ such that $u(t,\cdot)\in W^{1,q}_0(\Omega)$ for all $t\in (0,\tau]$ and
\[
    \left\{\begin{aligned}
        &\dot{u} + A(\cdot)u = f,\\
        &u(0) = \Phi u(\tau).    
    \end{aligned}\right.
\]
In particular, we deduce that the periodic ($\Phi=I$) and the antiperiodic ($\Phi=-I$) problem are well-posed.
\end{example}

\begin{rem}
The results of this section remain true if $A(\cdot)$ is merely strongly measurable and relatively continuous and $A(t)\in \cMR(X)$ for all $t\in [0,\tau]$, 
see~\cite{ACFP07}.

We mention that the Fredholm alternative is also used in~\cite{AR09}, where periodic non-autonomous problems are studied for maximal $L^p$-regularity.
\end{rem}

\section{Form methods: the autonomous case}\label{sec:3}

Let $H,V$ be Hilbert spaces over $\KK=\RR$ or $\CC$ be such that $V\dembed H$ (i.e., $V$ is a dense subspace of $H$ and there exists a $c_H>0$ such that $c_H\norm{v}_H\le\norm{v}_V$ for all $v\in V$). Then we identify $H$ with a dense subspace of $V'$ in the usual way ($x\in H$ is identified with the functional $v\mapsto \scalar{x}{v}_H$ in  $V'$). So we have the Gelfand triple $V\dembed H\dembed V'$.
Let $\form{a}\colon V\times V \to\KK$ be a continuous sesquilinear form that is \emphdef{quasi-coercive} in $H$; i.e., there exist $\omega\in\RR$, $\alpha>0$ such that
\[
    \Re\form{a}(v,v) + \omega\norm{v}_H^2\ge\alpha\norm{v}_V^2
\]
for all $v\in V$. We say that $\form{a}$ is \emphdef{coercive} if we can choose $\omega=0$.
Let $\cA\in\Linop(V,V')$ be defined by $\langle\cA v,w\rangle_{V',V} = \form{a}(v,w)$.
We call $\cA$ the \emphdef{operator on $V'$ associated with $\form{a}$}.
Then $-\cA$ generates a holomorphic $C_0$-semigroup on $V'$. 
The part $A$ of $\cA$ in $H$ is defined by
\begin{align*}
    D(A) &:= \{v\in V: \cA v\in H\},\\
    Av &:= \cA v.
\end{align*}
The semigroup generated by $-\cA$ leaves invariant the space $H$ and its restriction to $H$ is a holomorphic $C_0$-semigroup $(T(t))_{t\ge 0}$ whose generator is $-A$.

Frequently, the operator $A$ is the object of interest since it incorporates the boundary conditions. 
We call $A$ the \emphdef{operator on $H$ associated with $\form{a}$}.

\begin{example}
Let $\KK=\RR$, $V=H^1(a,b)$ and $H=L^2(a,b)$ with $-\infty<a<b<\infty$,
and define $\form{a}\colon V\times V \to\RR$ by
\[
    \form{a}(u,v) = \int_a^b u'v'.
\]
Then $D(A) = \{u\in H^2(a,b): u'(a)=u'(b)=0\}$, $Au=-u''$.
So $A$ is the negative Laplacian with Neumann boundary conditions on $(a,b)$.
\end{example}

Since $-A$ generates a holomorphic semigroup $T$, we know that $T*f\in H^1(0,\tau;H)$ for all $f\in L^2(0,\tau;H)$. Moreover, $w:=T*f$ is the unique solution $w\in H^1(0,\tau;H)\cap L^2(0,\tau;D(A))$ of
\[
    \left\{\begin{aligned}
        &\dot{w} + Aw = f, \\
        &w(0) = 0.
    \end{aligned}\right.
\]
If we want to study the problem where we look for $v\in H^1(0,\tau;H)\cap L^2(0,\tau;D(A))$ such that
\[
    \left\{\begin{aligned}
        &\dot{v}+Av = 0, \\
        &v(0) = x,
    \end{aligned}\right.
\]
then a solution exists if and only if $x\in D((A+\omega)^{1/2})$. In fact, one has
\begin{align*}
    \tr_2 &:= \{w(0) : w\in H^1(0,\tau;H)\cap L^2(0,\tau;D(A))\} \\
        &= (H,D(A))_{\frac{1}{2},2} \\
        &= [H,D(A)]_{\frac{1}{2}} \\
        &= D((A+\omega)^{1/2}).
\end{align*}
The previous to last identity between the real and the complex interpolation space holds since we are in Hilbert space. The last identity holds since the operator $A+\omega$ has the property BIP.
\begin{conclusion*}
$D((A+\omega)^{1/2})$ is the right space of initial values for which maximal regularity holds.
\end{conclusion*}

Thus it is an important task to identify this space. A counterexample of McIntosh~\cite{McI72} shows that $V\ne D((A+\omega)^{1/2})$ can occur. However, for elliptic operators with various boundary conditions the identity
\[
    D((A+\omega)^{1/2}) = V
\]
is known. For $\Omega=\RR^d$ this is the famous Kato conjecture, which was solved by Auscher, Hofmann, Lacey, McIntosh and Tchamitchian~\cite{AHLMcIT02}.
We refer to~\cite{EHMT23} for a beautiful introduction and further information.
Here we mention a very simple special case.

\begin{prop}
Assume that $\form{a} = \form{a}_1 + \form{a}_2$ with sesquilinear forms $\form{a}_1, \form{a}_2 \colon V\times V \to\KK$, where $\form{a}_1$ is continuous, symmetric and quasi-coercive in $H$ and $\form{a}_2$ satisfies
\[
    \abs{\form{a}_2(v,w)}\le M_2\norm{v}_V\norm{w}_H\quad\text{for all $v,w\in V$}.
\]
Then $\form{a}$ is quasi-coercive in $H$ and
\[
    D((A+\omega)^{1/2}) = V.
\]
\end{prop}
In the case $\form{a}_2=0$ the spectral theorem can be used for the proof. An easy perturbation argument then settles the general case.

Now we obtain the following special case of Proposition~\ref{prop:2.10new}.
Recall that $(T(t))_{t\ge 0}$ is the $C_0$-semigroup generated by $-A$ on $H$.

\begin{thm}\label{thm:3.3}
Suppose that the square root property $V=D((A+\omega)^{1/2})$ holds.
Let $\Phi\colon V\to V$ be linear such that $I-\Phi T(\tau)\colon V\to V$ is bijective.
Then for all $f\in L^2(0,\tau;H)$ there exists a unique $u\in H^1(0,\tau;H)\cap L^2(0,\tau;V)$ such that
\[
    \left\{\begin{aligned}
        &\dot{u}+Au = f, \\
        &u(0) = \Phi u(\tau).
    \end{aligned}\right.
\]
\end{thm}
Note that $T(\tau)H\subset V$. If $V\embed H$ is compact, then $\restrict{T(\tau)}{V}\colon V\to V$ is compact. 
So if $V\embed H$ is compact and $\Phi\in\Linop(V)$ is bounded, it suffices to show that $\Phi T(\tau)v=v$ implies $v=0$ for every $v\in V$ in order to deduce that $I-\Phi T(\tau)\colon V\to V$ is invertible.
We give an example to illustrate the results.

\begin{example}
Let $\Omega\subset\RR^d$ be open and bounded, $H=L^2(\Omega)$ and $V=H^1_0(\Omega)$ with $\KK=\RR$, and $\form{a}(v,w) =\int_\Omega \nabla v\cdot\nabla w$.
Then $Av=-\Delta v$ and $D(A)=\{v\in H^1_0(\Omega): \Delta v\in L^2(\Omega)\}$.
Denote by $\lambda_1>0$ the first eigenvalue of $A$. 
Let $\tau>0$ and let $\Phi\in\Linop(V)$.
Assume that
\begin{alenum}
\item\label{en:cond-a} $\Phi T(\tau)v = v$ implies $v=0$ whenever $v\in V$, or that
\item\label{en:cond-b} $\norm{\Phi}_{\Linop(V)} < e^{\lambda_1\tau}$.
\end{alenum}
Then for all $f\in L^2(0,\tau;H)$ there exists a unique $u\in H^1(0,\tau;H)\cap L^2(0,\tau;V)$ such that
\[
    \left\{\begin{aligned}
        &\dot{u} + Au = f, \\
        &u(0) = \Phi u(\tau).
    \end{aligned}\right.
\]
\begin{proof}
\ref{en:cond-a} suffices for the conclusion since the embedding $H^1_0(\Omega)\embed L^2(\Omega)$ is compact.

So suppose~\ref{en:cond-b}. Let $(e_n)_{n\in\NN}$ be an orthonormal basis of $H$ such that $e_n\in V$, $Ae_n=\lambda_n e_n$ and $0<\lambda_1\le \lambda_2\le\dots\le\lambda_n\le\lambda_{n+1}\le\dots$ for all $n\in\NN$.
Then 
\[
    \norm{v}_V^2 := \int_\Omega\abs{\nabla v}^2\dx = \sum_{n=1}^\infty\lambda_n\abs{\scalar{v}{e_n}_H}^2
\]
defines an equivalent norm on $V$ (see e.g.~\cite[Theorem~4.50]{AU23}).
Moreover, 
\[
    T(t)v = \sum_{n=1}^\infty e^{-\lambda_n t} \scalar{v}{e_n}_H e_n
\]
for all $v\in V$. Hence
\begin{align*}
    \norm{T(\tau)v}_V^2 &= \sum_{n=1}^\infty \lambda_n e^{-2\lambda_n \tau} \abs{\scalar{v}{e_n}_H}^2 \\
        &\le e^{-2\lambda_1\tau} \sum_{n=1}^\infty \lambda_n \abs{\scalar{v}{e_n}_H}^2 \\
        &= e^{-2\lambda_1\tau}\norm{v}_V^2.
\end{align*}
Thus $\norm{T(\tau)}_{\Linop(V)} \le e^{-\lambda_1\tau}$.
So by the assumption $\norm{\Phi T(\tau)}_{\Linop(V)}<1$ and
the claim follows from Theorem~\ref{thm:3.3}.
\end{proof}
\end{example}

\section{Non-autonomous equations}\label{sec:4}

Let $V$, $H$ be Hilbert spaces over $\KK=\RR$ or $\CC$ be such that $V\dembed H$; i.e., $V$ is a dense subspace of $H$ and there exists a $c_H>0$ such that
\[
    c_H\norm{v}_H\le\norm{v}_V\quad\text{for all $v\in V$.}
\] 
We consider a non-autonomous form $\form{a}\colon [0,\tau]\times V\times V\to\KK$; i.e.,
\begin{alenum}
\item \label{en:form-sesqui} $\form{a}(t,\cdot,\cdot)\colon V\times V\to\KK$ is a sesquilinear form for all $t\in[0,\tau]$,
\item \label{en:form-meas} $\form{a}(\cdot,v,w)\colon [0,\tau]\to\KK$ is measurable for all $v,w\in V$,
\item \label{en:form-cont} there exists an $M\ge 0$ such that $\abs{\form{a}(t,v,w)} \le M\norm{v}_V\norm{w}_V$ for all $v,w\in V$, $t\in[0,\tau]$.
\end{alenum}
We assume furthermore that the form is \emphdef{quasi-coercive}; i.e.,
\begin{alenum}[resume*]
\item\label{en:form-quasi-coercive} there exist $\omega\in\RR$ and $\alpha>0$ such that $\Re \form{a}(t,v,v)+\omega\norm{v}_H^2 \ge\alpha\norm{v}_V^2$ for all $v\in V$, $t\in[0,\tau]$. If we can choose $\omega=0$, we call the form \emphdef{coercive}.
\end{alenum}

For $t\in[0,\tau]$ let $\cA(t)\in\Linop(V,V')$ be given by
\[
    \langle \cA(t)v,w\rangle_{V',V} := \form{a}(t,v,w).
\]
Let
\[
    \MR(V,V') = H^1(0,\tau;V')\cap L^2(0,\tau;V).
\]
Then $\MR(V,V')\subset C([0,\tau];H)$, see~\cite[Proposition~III.1.2 p.\,106]{Sho97}.

The following result is due to J.\,L.~Lions.
\begin{thm}[Lions]\label{thm:lions}
Let $f\in L^2(0,\tau; V')$, $x\in H$. Then there exists a unique $u\in\MR(V,V')$ such that
\[
    \Pini[f,x]\quad\left\{
    \begin{aligned}
    &\dot{u}(t) + \cA(t)u(t) = f(t)\quad \text{a.e.\ and}\\
    &u(0) = x.
    \end{aligned}
    \right.
\]
\end{thm}

This can be proved by an elegant use of Lions' representation theorem, see~\cite[p.\,112]{Sho97}. We will use similar arguments as in Section~\ref{sec:2} in order to extend Theorem~\ref{thm:lions} to generalized time--boundary conditions.

Let $\Phi\colon H\to H$ be linear. We study the following problem.
\[
    \PPhi[f]\quad \begin{cases}
    \text{Given $f\in L^2(0,\tau;V')$ find $u\in\MR(V',V)$ such that} \\
    \begin{aligned}
    &\dot{u}(t) + \cA(t)u(t) = f(t) \quad\text{a.e.\ and}\\
    &u(0)=\Phi u(\tau).
    \end{aligned}
    \end{cases}
\]
We say that $\PPhi$ is \emphdef{well-posed} if for all $f\in L^2(0,\tau;V')$ there exists a unique $u\in\MR(V,V')$ solving $\PPhi[f]$. In that case it follows from the closed graph theorem that
\[
    \norm{u}_{\MR(V,V')} \le c\norm{f}_{L^2(0,\tau;V')}
\]  
for some constant $c$ independent of $f$, where
\[
    \norm{u}_{\MR(V,V')}^2 := \int_\Omega\norm{\dot{u}(t)}_{V'}^2\dx[t] + \int_\Omega\norm{u(t)}_V^2\dx[t].
\]

\begin{thm}\label{thm:4.2}
Assume that 
\[
    \norm{\Phi}_{\Linop(H)} < (1+2\alpha c_H\tau)^{\frac{1}{2}} e^{-\omega \tau}.
\]
Then $\PPhi$ is well-posed.
\end{thm}
For $\norm{\Phi}_{\Linop(H)}\le 1$ and $\omega=0$ this result was proved in~\cite[Theorem~5.4]{ACE23} with the help of Lions' representation theorem. In the case where $\form{a}(t,\cdot,\cdot)$ is symmetric and $\norm{\Phi}_{\Linop(H)}\le 1$ it is due to Showalter~\cite[Proposition~II.2.4]{Sho97}.
In the case where $\omega=0$ and $\norm{\Phi}_{\Linop(H)}\le 1$, a time--space discretization with optimal error estimates for the problem $\PPhi$ is established in~\cite{ACE24}, 
which also yields a proof of Theorem~\ref{thm:4.2} in this special case, see~\cite[Theorem~3.3]{ACE24}.

\begin{rem*}
Our result shows that given $\Phi\in\Linop(H)$, if $\form{a}(t,\cdot,\cdot)$ is defined on $(0,\infty)$ (satisfying~\ref{en:form-sesqui}--\ref{en:form-quasi-coercive} above), then $\PPhi$ is well-posed if $\tau>0$ is large enough.

In the following we keep $\tau>0$ fixed again.
\end{rem*}

We recall the following integration-by-parts formula, see~\cite[Corollary~III.1.1, p.\,106]{Sho97}.
\begin{lem}\label{lem:ibp}
Let $v,w\in\MR(V,V')$. Then
\[
    \scalar{v(s)}{w(s)}_H-\scalar{v(r)}{w(r)}_H = 
    \int_r^s\langle v'(t),w(t)\rangle_{V',V}\dx[t] + \int_r^s\langle w'(t),v(t)\rangle_{V',V}\dx[t]
\]
for $0\le r<s\le\tau$.
\end{lem}

By Theorem~\ref{thm:lions}, given $x\in H$ there exists a unique $v\in\MR(V,V')$ solving the homogeneous problem $\Pini[0,x]$.
We define the linear operator $T(\tau)\colon H\to H$ by $T(\tau)x=v(\tau)$, where $v\in\MR(V,V')$ is the solution of $\Pini[0,x]$. Of course, here in a non-autonomous setting $T(\tau)$ is not a semigroup operator. The notation is just by analogy to Section~\ref{sec:2}.

\begin{lem}\label{lem:4.4}
One has $T(\tau)\in\Linop(H)$ and 
\[
    \norm{T(\tau)}_{\Linop(H)} \le \frac{1}{(1+2\alpha c_H\tau)^\frac{1}{2}} e^{\omega \tau}.
\]
\end{lem}
\begin{proof}
\begin{parenum}
\item\label{en:part1} Assume that $\omega=0$. Let $x\in H$, $v$ the solution of $\Pini[0,x]$. Then for $0\le r<s\le\tau$
\begin{equation}\label{eq:4.2}
    \int_r^s\langle v'(t),v(t)\rangle_{V',V}\dx[t] + \int_r^s\form{a}(t,v(t),v(t))\dx[t] = 0.
\end{equation}
By Lemma~\ref{lem:ibp},
\[
    \frac{1}{2}\norm{v(s)}_H^2 - \frac{1}{2}\norm{v(r)}_H^2 = \int_r^s\langle v'(t),v(t)\rangle_{V',V} \dx[t].
\]
Since $\form{a}(t,v(t),v(t))\ge 0$ we conclude from~\eqref{eq:4.2} that $\norm{v(\cdot)}_H^2$ is decreasing. Hence
\[
    \int_0^\tau \form{a}(t,v(t),v(t))\dx[t] \ge \alpha \int_0^\tau \norm{v(t)}_V^2\dx[t] 
        \ge \alpha c_H\int_0^\tau \norm{v(t)}_H^2\dx[t] \ge \alpha c_H\tau\norm{v(\tau)}_H^2.
\]
By~\eqref{eq:4.2},
\[
    \frac{1}{2}\norm{v(\tau)}_H^2 - \frac{1}{2}\norm{v(0)}_H^2 + \alpha c_H\tau\norm{v(\tau)}_H^2 \le 0.
\]
Hence $(1+2\alpha c_H\tau)\norm{v(\tau)}_H^2\le \norm{v(0)}_H^2$.
This proves the claim if $\omega=0$.

\item Assume that $\omega\in\RR$. Let $v\in\MR(V,V')$ be such that $\dot{v}+\cA(\cdot) v=0$. Let $u(t)=e^{-\omega t}v(t)$. Then $\dot{u}+(\cA(\cdot)+\omega)u=0$.
It follows from~\ref{en:part1} that $\norm{u(\tau)}_H\le (1+2\alpha c_H\tau)^{-\frac{1}{2}}\norm{u(0)}_H$. Since $v(0)=u(0)$ and $v(\tau)=e^{\omega\tau}u(\tau)$, the claim is proved.\qedhere
\end{parenum}
\end{proof}

\begin{thm}\label{thm:4.5}
Let $\Phi\colon H\to H$ be linear. The following statements are equivalent.
\begin{romanenum}
\item $\PPhi$ is well-posed.
\item $I-\Phi T(\tau)\colon H\to H$ is bijective.
\end{romanenum}
\end{thm}
\begin{proof}
\begin{parenum}
\item[(ii)$\Rightarrow$(i).] Let $f\in L^2(0,\tau;V')$. Let $w\in\MR(V,V')$ be such that 
$\dot{w}+\cA(\cdot) w=f$, $w(0)=0$. 
By the assumption there exists $x\in H$ such that $x-\Phi T(\tau)x = \Phi w(\tau)$. 
Let $v\in\MR(V,V')$ be such that $\dot{v}+\cA(\cdot) v=0$, $v(0)=x$. Then $u:=w+v$ is a solution of $\PPhi[f]$. We have shown existence.

In order to prove uniqueness, let $u\in\MR(V,V')$ be such that $\dot{u}+\cA(\cdot) u=0$, $u(0)=\Phi u(\tau)$.  Let $x:=u(0)$. Then $\Phi T(\tau)x=x$ by the definition of $T(\tau)$. Hence $x=0$.
Theorem~\ref{thm:lions} implies that $u=0$.

\item[(i)$\Rightarrow$(ii).] We show that $I-\Phi T(\tau)$ is surjective.
Let $y\in H$. There exists $w\in\MR(V,V')$ such that $w(\tau)=T(\tau)y$ and $w(0)=0$.
In fact, there exists a $w_1\in\MR(V,V')$ such that $w_1(0)=T(\tau)y$.
Let $\phi\in C^\infty[0,\tau]$ be such that $\phi(0)=0$ and $\phi(\tau)=1$. 
Define $w(t)=\phi(t) w_1(\tau-t)$.
Let $f=\dot{w}+\cA(\cdot) w$. By hypothesis there exists a $u\in\MR(V,V')$ such that $\dot{u}+\cA(\cdot) u=f$, $u(0)=\Phi u(\tau)$. Let $x:=u(0)$, $v=u-w\in\MR(V,V')$.
Then $\dot{v}+\cA(\cdot) v=0$, $v(0)=x$. Thus 
\[
    T(\tau)x = v(\tau) = u(\tau)-w(\tau)=u(\tau)-T(\tau)y
\]
and therefore $x=\Phi u(\tau) = \Phi T(\tau)x + \Phi T(\tau)y$.
We have shown that
\[
    (I-\Phi T(\tau))x = \Phi T(\tau)y.
\]
Hence $(I-\Phi T(\tau))(x+y)=y$.
We have proved surjectivity.

To show injectivity, let $x\in H$ be such that $x=\Phi T(\tau)x$.
Let $v$ be the solution of $\Pini[0,x]$. Then $v(\tau)=T(\tau)x$ by definition.
Hence $v$ is a solution of $\PPhi[0]$.
Hence $v=0$, and so $x=0$.\qedhere
\end{parenum}
\end{proof}

\begin{proof}[Proof of Theorem~\ref{thm:4.2}]
Assume that
\[
    \norm{\Phi}_{\Linop(H)} < (1+2\alpha c_H\tau)^{\frac{1}{2}}e^{-\omega\tau}.
\]
Then by Lemma~\ref{lem:4.4}, 
\[
    \norm{\Phi T(\tau)}_{\Linop(H)} < (1+2\alpha c_H\tau)^{\frac{1}{2}}e^{-\omega\tau}\norm{T(\tau)}_{\Linop(H)}\le 1.
\]
Thus, $I-\Phi T(\tau)$ is invertible and the result follows from Theorem~\ref{thm:4.5}.
\end{proof}

\begin{rem}
Using the embedding $V\embed H$, one can exchange some $\alpha$ for $\omega$ in the quasi-coerciveness condition~\ref{en:form-quasi-coercive}.
This allows to further optimize the estimate in Lemma~\ref{lem:4.4} and therefore the required bound in Theorem~\ref{thm:4.2} in dependence of $c_H$, $\alpha$ and $\tau$.

It is a straightforward calculation to check that the following choices are optimal.
\begin{itemize}
\item
If $c_H\ge 1$, it is optimal to put all of $\alpha$ into $\omega$, i.e.\ in Theorem~\ref{thm:4.2} it suffices to assume
\[
    \norm{\Phi}_{\Linop(H)}<e^{(c_H^2\alpha -\omega)\tau}.
\]
\item
If $\frac{1-c_H}{2\alpha c_H^2}\ge\tau$ (which implies $c_H<1$), it is optimal to keep all of $\alpha$. So the bound in Theorem~\ref{thm:4.2} remains the same.

\item
Finally, if $c_H<1$ and $\frac{1-c_H}{2\alpha c_H^2}<\tau$, the optimized bound in Theorem~\ref{thm:4.2} takes the form
\[
    \norm{\Phi}_{\Linop(H)} < \frac{1}{\sqrt{c_H}}e^{c_H^2\alpha\tau - (1-c_H)/2 - \omega\tau}.
\]
\end{itemize}
\end{rem}

Actually, Theorem~\ref{thm:4.2} remains true if $\Phi\colon H\to H$ is nonlinear and Lipschitz continuous.
We state this as a new result.
\begin{thm}
Let $\Phi\colon H\to H$ be a map satisfying 
\[
    \norm{\Phi(x)-\Phi(y)}_H\le L\norm{x-y}_H
\]
for all $x,y\in H$ where $0<L<(1+2\alpha c_H\tau)^{\frac{1}{2}}e^{-\omega\tau}$.
Then for all $f\in L^2(0,\tau;V')$ there exists a unique $u\in\MR(V,V')$ solving
\[
    \PPhi[f]\quad\left\{\begin{aligned}
        &\dot{u}+\cA(\cdot) u = f,\\
        &u(0) = \Phi(u(\tau)).
    \end{aligned}\right. 
\]
\end{thm}
\begin{proof}
Fix $f\in L^2(0,\tau;V')$. For $x\in H$, let $Sx:=\Phi(u(\tau))$, where $u\in\MR(V,V')$ is the solution of $\dot{u}+\cA(\cdot) u=f$, $u(0)=x$. If $y\in H$ is a further initial value and $w\in\MR(V,V')$ the solution of $\dot{w}+\cA(\cdot) w=f$, $w(0)=y$, then $v:=u-w\in\MR(V,V')$, $\dot{v}+\cA(\cdot) v=0$ and $v(0)=x-y$.
Thus $v(\tau)=T(\tau)(x-y)$. By Lemma~\ref{lem:4.4},
\[
    \norm{v(\tau)}_H\le\frac{1}{(1+2\alpha c_H\tau)^{1/2}}e^{\omega\tau}\norm{x-y}_H.
\]
Moreover,
\begin{align*}
    \norm{Sx-Sy}_H &= \norm{\Phi(u(\tau)) - \Phi(w(\tau))}_H \\
        &\le L\norm{u(\tau)-w(\tau)}_H = L\norm{v(\tau)}_H \\
        &\le \frac{Le^{\omega\tau}}{(1+2\alpha c_H\tau)^{1/2}}\norm{x-y}_H.
\end{align*}
By Banach's fixed point argument, there exists a unique $x\in H$ such that $Sx=x$.
Let $u\in\MR(V,V')$ be such that $\dot{u}+\cA(\cdot) u =f$, $u(0)=x$. 
Then $u(0)=x=Sx=\Phi(u(\tau))$. 

In order to show uniqueness, let $u_1,u_2\in\MR(V,V')$ be two solutions of $\PPhi[f]$.
Then $Su_j(0)=u_j(0)$ for $j=1,2$. Hence $u_1(0)=u_2(0)$.
Let $v=u_1-u_2$. Then $v\in\MR(V,V')$, $\dot{v} + \cA(\cdot) v=0$, $v(0)=0$. It follows from Theorem~\ref{thm:lions} that $v=0$.
\end{proof}

\begin{rem}
If in addition to the assumptions of this section the map $\cA\colon [0,\tau]\to\Linop(V,V')$ is continuous, then we can apply Theorem~3.3 with $X=V'$ and $D=V$.
Then for $1<p<\infty$, $\tr_p=(V',V)_{\frac{1}{p'},p}$ and given a linear $\Phi\colon\tr_p\to\tr_p$ such that $I-\Phi T(\tau)\colon \tr_p\to\tr_p$ is bijective, for all $f\in L^p(0,\tau;V')$ there exists a unique $u\in\MR_p$ such that $\dot{u}+\cA(\cdot)u=f$, $u(0)=\Phi u(\tau)$.

If $\cA \colon [0,\tau]\to\Linop(V,V')$ is not continuous, this result may fail for any $p\in(1,\infty)\setminus\{2\}$ even for $\Phi=0$. In fact, 
a recent result by Bechtle, Mooney and Veraar~\cite[Theorem~2.5]{BMV24} shows 
this even if $\cA(t)$ is an elliptic operator of second order.
\end{rem}

Finally, we mention a semigroup property.
\begin{thm}\label{thm:5.8}
Assume that $\omega=0$; i.e., the form $\form{a}\colon [0,\tau]\times V\times V \to\KK$ is coercive.
Let $\Phi\in\Linop(H)$ be such that $\norm{\Phi}_{\Linop(H)}\le 1$.

Define the operator $C$ in $\cH := L^2(0,\tau;H)$ by
\[
    \begin{aligned}
    D(C) &:= \{u\in\MR(V,V') : \dot{u} + \cA(\cdot) u\in\cH,\ u(0)=\Phi u(\tau)\}, \\
    C u &:= \dot{u} + \cA(\cdot)u.
    \end{aligned}
\]
Then $C$ is m-accretive.
\end{thm}
\begin{proof}
\begin{parenum}
\item $C$ is accretive. Let $u\in D(C)$. Then 
\begin{align*}
    \Re\scalar{Cu}{u}_{\cH} &= \Re\int_0^\tau \scalar{ (Cu)(t)}{u(t)}_H\dx[t] \\
    &=\Re\int_0^\tau \langle (Cu)(t),u(t)\rangle_{V',V} \dx[t] \\
    &=\Re\int_0^\tau\langle \dot{u}(t),u(t)\rangle_{V',V}\dx[t] + \Re\int_0^\tau \form{a}(t,u(t),u(t))\dx[t] \\
    &\ge \frac{1}{2}\norm{u(\tau)}_H^2 - \frac{1}{2}\norm{u(0)}_H^2 \\
    &=\frac{1}{2}\Bigl(\norm{u(\tau)}_H^2 - \norm{\Phi u(\tau)}_H^2\Bigr) \ge 0.
\end{align*}
\item Theorem~\ref{thm:4.2} implies that $0\in\rho(C)$ (the resolvent set of $C$).
Since $\rho(C)$ is open, $C+\lambda I$ is surjective for $\lambda>0$ small enough.
This implies that $C$ is m-accretive.
\qedhere
\end{parenum}
\end{proof}

\begin{cor}\label{cor:5n.10}
Let $\Phi\in\Linop(H)$ be such that $\norm{\Phi}_{\Linop(H)}\le 1$. 
Define the operator $B$ in $\cH=L^2(0,\tau;H)$ by
\[
    \begin{aligned}
    D(B) &:= \{u\in H^1(0,\tau;H) : u(0)=\Phi u(\tau)\},\\
    B u &:= u'.
    \end{aligned}
\]
Then $B$ is m-accretive.
\end{cor}
\begin{proof}
Accretivity follows from the proof of Theorem~\ref{thm:5.8}.
If $\lambda>0$, then $B+\lambda I$ is surjective.
This follows from Theorem~\ref{thm:5.8} by considering $\form{a}(t,v,w) := \lambda \scalar{v}{w}_H$ for $v,w\in V$ and $t\in [0,\tau]$.
\end{proof}

Let $\cV:= L^2(0,\tau;V)$ and $\cH := L^2(0,\tau;H)$.
Then $\cV\dembed\cH\dembed\cV'$ is a Gelfand triple where $\cV'=L^2(0,\tau;V')$.
The form $\widetilde{\form{a}}\colon \cV\times\cV\to\KK$ given by $\widetilde{\form{a}}(v,w) = \int_0^\tau\form{a}(t,v(t),w(t))\dx[t]$ is continuous and coercive if we assume that $\omega=0$, what we do in the remainder of this section.
Let $\widetilde{\cA}\colon\cV\to\cV'$ be given by $\langle \widetilde{\cA}v,w\rangle = \widetilde{\form{a}}(v,w)$.
Then $(\widetilde{\cA}v)(t) = \cA(t)v(t)$~a.e.\ for all $v\in\cV$.
The operator $\widetilde{\cA}$ is invertible since $\Re\langle\widetilde{\cA}v,v\rangle \ge\alpha\norm{v}_{\cV}^2$ for all $v\in V$.
Let $\widetilde{A}$ be the part of $\widetilde{\cA}$ in $\cH$. Then $\widetilde{A}$ is m-accretive and invertible.
In fact, $-\widetilde{A}$ generates a holomorphic $C_0$-semigroup on $\cH$.
By definition,
\[
    \begin{aligned}
    D(\widetilde{A}) &= \{v\in\cH : \text{$v(t)\in D(A(t))$ a.e.\ and $t\mapsto A(t)v(t)\in\cH$}\},\\
    (\widetilde{A}v)(t) &= A(t)v(t)\text{ a.e.}
    \end{aligned}    
\]
So with $B$ and $C$ as in Theorem~\ref{thm:5.8} and Corollary~\ref{cor:5n.10}, one has $B+\widetilde{A}\subset C$ (i.e., $D(B)\cap D(\widetilde{A})\subset D(C)$ and $Cu=Bu+\widetilde{A}u$ for $u\in B(B)\cap D(\widetilde{A})$),
but it may happen that $D(C)\supsetneq D(B)\cap D(\widetilde{A})$.
If $D(C)=D(B)\cap D(\widetilde{A})$, we speak of \emphdef{maximal regularity in $H$}. We will investigate this property in the next section.

The investigation of conditions implying that the sum of two sectorial operators is closed goes back to the seminal paper~\cite{DPG75} of Da~Prato and Grisvard. 
It is a guideline in the theory of maximal regularity.
The theorem of Dore and Venni~\cite{DV87} gives a very transparent condition in the commutative case.
For a non-commutative situation as we have it here, see the paper by Monniaux and Prüss~\cite{MP97}.
In the latter paper, in addition to the condition of bounded imaginary powers, a commutator condition is imposed.

\section{Maximal regularity in \texorpdfstring{$H$}{H}}\label{sec:6}

Lions' theorem (Theorem~\ref{thm:lions}) gives a well-posedness result under very general assumptions, even with maximal regularity in $L^2(0,\tau;V')$: The terms $\dot{u}$, $\cA(\cdot)u$ and $f$ are all three in $L^2(0,\tau;V')$.
However, as pointed out in Section~\ref{sec:3}, in general the operator $\cA(t)$ by itself does not incorporate the boundary conditions, and its part $A(t)$ in $H$ is much more interesting.
Unfortunately, for these operators maximal regularity in $H$ fails in general.
It is merely valid under additional assumptions, see the survey~\cite{ADF17}.
This means, in the situation of Theorem~\ref{thm:lions}, there may exist $f\in L^2(0,\tau;H)$ such that the solution $u$ of $\Pini[f,0]$ is not in $H^1(0,\tau;H)$ and consequently, the property that $u(t)\in D(A(t))$~a.e.\ and $u\in L^2(0,\tau;D(A))$ is violated.
By a result of Fackler~\cite{Fac17}, this can happen even if the function $\cA\colon [0,\tau]\to\Linop(V,V')$
is continuous and $\form{a}(t,\cdot,\cdot)$ is symmetric.
Even though the form domain is constant, the domains $D(A(t))$ are time-dependent, in general.
For this reason we consider more restrictive assumptions on the form $\form{a}$ than in Section~\ref{sec:4}.

We start by a known positive result for the initial value case, Theorem~\ref{thm:5.1}\,(b).
Under the hypotheses made in Theorem~\ref{thm:5.1} and following the same strategy as in Section~\ref{sec:4}, we are able to characterize maximal regularity in $H$ for the generalized boundary condition $u(0)=\Phi u(\tau)$.

In this section we make the following assumptions. $V$ and $H$ are Hilbert spaces over $\KK=\RR$ or $\CC$ such that $V\dembed H$. Let $c_H>0$ be such that $c_H\norm{v}_H\le\norm{v}_V$ for all $v\in V$. We identify $H$ with a subspace of $V'$ in the usual way.
Let $\form{a}_1,\form{a}_2\colon[0,\tau]\times V\times V\to\KK$ be non-autonomous forms; i.e., $\form{a}_1,\form{a}_2$ satisfy~\ref{en:form-sesqui},~\ref{en:form-meas},~\ref{en:form-cont} from Section~\ref{sec:4}.

In addition we assume that
\begin{alenum}[resume]
\item $\form{a}_1(t,v,w) = \conj{\form{a}_1(t,w,v)}$ for all $w,v\in V$ and $t\in[0,\tau]$\quad (symmetry);

\item there exists an $\alpha>0$ such that $\form{a}_1(t,v,v)\ge \alpha\norm{v}_V^2$ for all $v\in V$ and $t\in[0,\tau]$\quad (coerciveness);

\item there exists an $L\ge 0$ such that $\abs{\form{a}_1(t,v,w)-\form{a}_1(s,v,w)}\le L\abs{t-s}\cdot\norm{v}_V\norm{w}_V$ for all $v,w\in V$ and $t,s\in[0,\tau]$\quad(Lipschitz continuity);

\item there exists an $M\ge 0$ such that $\abs{\form{a}_2(t,v,w)}\le M\norm{v}_V\norm{w}_H$ for all $v,w\in V$ and $t\in[0,\tau]$.
\end{alenum}

We define $\form{a}(t,v,w) = \form{a}_1(t,v,w)+\form{a}_2(t,v,w)$ for all $v,w\in V$ and $t\in[0,\tau]$, and the operator $\cA(t)\in\Linop(V,V')$ is given by $\langle\cA(t) v,w\rangle_{V',V} = \form{a}(t,v,w)$.

It is easy to see that the form $\form{a}$ is \emphdef{quasi-coercive}; i.e., there exist $\omega\in\RR$ and $\alpha>0$ such that 
\[
    \Re\form{a}(t,v,v)+\omega\norm{v}_H^2\ge\alpha\norm{v}_V^2
    \quad\text{for all $v\in V$ and $t\in[0,\tau]$.}
\]
Thus Theorem~\ref{thm:lions} applies.
But under the more restrictive hypotheses of this section we have even maximal regularity in $H$. By this we mean the following.
Let $\MR(V,H) := H^1(0,\tau;H)\cap L^2(0,\tau;V)$. Then by~\cite[Corollary~3.3 and Theorem~4.2]{ADLO14} the following result is true.

\begin{thm}\leavevmode\label{thm:5.1}
\begin{alenum}
\item If $u\in\MR(V,H)$ satisfies $\cA(\cdot)u \in L^2(0,\tau;H)$, then $u\in C([0,\tau];V)$.
\item For all $f\in L^2(0,\tau; H)$ and all $x\in V$ there exists a unique $u\in\MR(V,H)$ such that
\[
    \left\{
    \begin{aligned}
    & \dot{u}(t)+\cA(t)u(t)=f(t)\quad\text{a.e. and}\\
    & u(0)=x.
    \end{aligned}\right.
\]
\end{alenum}
\end{thm}

Since in the situation of Theorem~\ref{thm:5.1}\,(b) one has $\cA(t)u(t)=f(t)-\dot{u}(t)\in H$ a.e.,
we may rewrite the equation as $u(t)\in D(A(t))$ a.e.\ and $\dot{u}(t)+A(t)u(t)=f(t)$ a.e.,
where $A(t)$ denotes the part of $\cA(t)$ in $H$.
We also mention that for each $t\in[0,\tau]$, the form $\form{a}(t,\cdot,\cdot)=\form{a}_1(t,\cdot,\cdot) + \form{a}_2(t,\cdot,\cdot)$ satisfies the square root property; i.e., one has
\[
    V = D\bigl((A(t)+\omega)^{1/2}\bigr).
\]
So things are consistent with our comments in Section~\ref{sec:3}.

Now we let $\Phi\colon V\to H$ be linear and we consider the boundary condition $u(0)=\Phi u(\tau)$.
Using the technique of Section~\ref{sec:4}, we can prove the following result on well-posedness in $H$.
We define by $S(\tau)\colon V\to V$ the linear mapping $S(\tau)x=v(\tau)$, where for $x\in V$, $v\in\MR(V,H)$ is the unique solution of the homogeneous problem $\dot{v}+\cA(\cdot) v=0$, $v(0)=x$.
Thus $S(\tau)=\restrict{T(\tau)}{V}$, where $T(\tau)$ has been defined in Section~\ref{sec:4}.

\begin{thm}\label{thm:5.2}
Let $\Phi\colon V\to V$ be linear. The following statements are equivalent.
\begin{romanenum}
\item $I-\Phi S(\tau)\colon V\to V$ is bijective.
\item For all $f\in L^2(0,\tau;H)$ there exists a unique $u\in\MR(V,H)$ such that
\[
    \left\{
    \begin{aligned}
    & \dot{u}(t)+\cA(t)u(t)=f(t)\quad\text{a.e. and}\\
    & u(0)=\Phi u(\tau).
    \end{aligned}\right.
\]
\end{romanenum}
\end{thm}

We start with the following lemma on traces.
\begin{lem}\label{lem:5.3}
Let $y\in V$. Then there exists $w\in\MR(V,H)$ such that $\dot{w}+\cA(\cdot)w\in L^2(0,\tau;H)$, $w(0)=0$ and $w(\tau)=y$.
\end{lem}
\begin{proof}
Consider $\widetilde{\cA}(t)=\cA(\tau-t)$. Applying Theorem~\ref{thm:5.1}\,(b) to $\widetilde{\cA}$ instead of $\cA$ we find $w_1\in\MR(V,H)$ such that $\dot{w}_1+\widetilde{\cA}(\cdot)w_1=0$, $w_1(0)=y$. Let $w_2(t)=w_1(\tau-t)$. Then $w_2(\tau)=y$ and 
\[
    \dot{w}_2(t)+\cA(t)w_2(t)=-\dot{w}_1(\tau-t)+\widetilde{\cA}(\tau-t)w_1(\tau-t).
\]
Thus $w_2\in\MR(V,H)$ and $\dot{w}_2+\cA(\cdot)w_2\in L^2(0,\tau;H)$.
Now choose $\phi\in C^\infty[0,\tau]$ such that $\phi(0)=0$ and $\phi(\tau)=1$. Then $w=\phi w_2\in\MR(V,H)$ and
\[
    \dot{w}(t)+\cA(t)w(t) = \phi(t)\bigl(\dot{w}_2(t)+\cA(t)w_2(t)\bigr) + \dot{\phi}(t)w_2(t).
\]
So $\dot{w}+\cA(\cdot) w\in L^2(0,\tau;H)$, $w(0)=0$ and $w(\tau)=y$.
\end{proof}

\begin{proof}[Proof of Theorem~\ref{thm:5.2}]
\begin{parenum}
\item[(i)$\Rightarrow$(ii).] Let $w\in\MR(V,H)$ be the solution of $\dot{w}+\cA(\cdot) w=f$, $w(0)=0$.
By the assumption there exists a unique $x\in V$ such that $x-\Phi S(\tau)x=\Phi w(\tau)$.
Let $v\in\MR(V,H)$ be such that $\dot{v}+\cA(\cdot) v=0$, $v(0)=x$. Then $S(\tau)x=v(\tau)$.
Moreover, $u=v+w\in\MR(V,H)$, $\dot{u}+\cA(\cdot) u=f$ and $u(0)=\Phi u(\tau)$. This shows existence.
Uniqueness follows as in Theorem~\ref{thm:4.5}.

\item[(ii)$\Rightarrow$(i).] We show that $I-\Phi S(\tau)\colon V\to V$ is surjective.
Let $y\in V$. Then $S(\tau)y\in V$. By Lemma~\ref{lem:5.3} there exists $w\in\MR(V,H)$ such that $\dot{w}+\cA(\cdot) w\in L^2(0,\tau; H)$, $w(0)=0$ and $w(\tau)=S(\tau)y$.
By assumption there exists $u\in\MR(V,H)$ such that $\dot{u}+\cA(\cdot) u = \dot{w}+\cA(\cdot) w$ and $u(0)=\Phi u(\tau)$. 
Let $v=u-w$. 
Then $\dot{v}+\cA(\cdot) v=0$ and $v(0)=u(0)\eqqcolon x\in V$, $S(\tau)x=v(\tau)=u(\tau)-w(\tau) = u(\tau) - S(\tau)y$.
Hence $x-\Phi S(\tau)x = \Phi S(\tau) y$.
This implies that $(I-\Phi S(\tau))(x+y)=y$.
We have proved surjectivity. The proof of injectivity is the same as in Theorem~\ref{thm:4.5}.\qedhere
\end{parenum}
\end{proof}

If $\Phi$ is not assumed to map into $V$, we instead obtain the next characterization.
\begin{thm}\label{thm:5.4new}
Let $\Phi\colon V\to H$ be linear. The following statements are equivalent.
\begin{romanenum}
\item For all $f\in L^2(0,\tau;H)$ there exists a unique $u\in\MR(V,H)$ such that
\[
    \left\{\begin{aligned}
        &\dot{u}+\cA(\cdot)u = f, \\
        &u(0)=\Phi u(\tau).
    \end{aligned}\right.
\]
\item $\Phi$ satisfies the following two properties.
\begin{alenum}
\item $\Phi S(\tau)x=x$ implies $x=0$, for every $x\in V$.
\item $\Phi V\subset (I-\Phi S(\tau))V$.
\end{alenum}
\end{romanenum}
\begin{proof}
\begin{parenum}
\item[(i)$\Rightarrow$(ii).] In order to show~(a), suppose that $x\in V$ with $\Phi S(\tau)x=x$.
By Theorem~\ref{thm:5.1}\,(b) there exists a $v\in\MR(V,H)$ such that $\dot{v}+\cA(\cdot)v=0$, $v(0)=x$. Then $v(\tau)=S(\tau)x$ and therefore $\Phi v(\tau)=v(0)$. By~(i) it follows that $v=0$. Hence $x=0$.

Next we show~(b). Let $y\in V$. By Lemma~\ref{lem:5.3} there exists a $w\in\MR(V,H)$ such that $f:=\dot{w}+\cA(\cdot) w\in L^2(0,\tau;H)$, $w(0)=0$ and $w(\tau)=y$.
Let $u\in\MR(V,H)$ be such that $\dot{u}+\cA(\cdot) u = f$ and $\Phi u(\tau)=u(0)=: x$.
Then $v=u-w\in\MR(V,H)$, $\dot{v}+\cA(\cdot) v=0$, $v(0)=x$ and $v(\tau)=u(\tau)-w(\tau) = u(\tau)-y$.
Hence $S(\tau)x=u(\tau)-y$. Thus $x=\Phi u(\tau) = \Phi S(\tau)x+\Phi y$, i.e.~$x-\Phi S(\tau)x=\Phi y$.

\item[(ii)$\Rightarrow$(i).]
Clearly~(a) implies uniqueness. In order to prove existence, let $f\in L^2(0,\tau;H)$. By Theorem~\ref{thm:5.1}\,(b) there exists a $w\in\MR(V,H)$ such that $\dot{w}+\cA(\cdot) w = f$, $w(0)=0$.
Then $y:=w(\tau)\in V$. By~(b) there exists an $x\in V$ such that $x-\Phi S(\tau)x=\Phi y$.
Let $v\in\MR(V,H)$ be such that $\dot{v}+\cA(\cdot)v=0$, $v(0)=x$. Then $v(\tau)=S(\tau)x$ and $\Phi v(\tau) = x-\Phi y$. Let $u=v+w$. Then $\dot{u}+\cA(\cdot)u=f$, $u(0)=x$, $u(\tau)=v(\tau)+y$ and therefore $\Phi u(\tau) = x-\Phi y + \Phi y = u(0)$.\qedhere
\end{parenum}
\end{proof}

\end{thm}

In the case where $\PPhi$ is weakly well-posed, Theorem~\ref{thm:5.4new} can be reformulated.
Recall that for $x\in H$, $T(\tau)x:=v(\tau)$, where $v\in\MR(V,V')$ is the solution of $\dot{v}+\cA(\cdot) v = 0$, $v(0)=x$.

\begin{prop}\label{prop:5.5new}
Let $\Phi\colon H\to H$ be linear. Assume that $I-\Phi T(\tau)\colon H\to H$ is bijective. Consider the following conditions.
\begin{romanenum}
\item\label{en:i-5.5n} For all $f\in L^2(0,\tau;H)$, the solution of $\PPhi[f]$ is in $\MR(V,H)$.
\item\label{en:ii-5.5n} $(I-\Phi T(\tau))^{-1}V\subset V$.
\item\label{en:iii-5.5n} $\Phi V\subset V$ and $(I-\Phi T(\tau))^{-1}V\subset V$.
\end{romanenum}
Then~\ref{en:i-5.5n} implies~\ref{en:ii-5.5n}, and~\ref{en:iii-5.5n} implies~\ref{en:i-5.5n}.
\end{prop}
\begin{proof}
\begin{parenum}
\item[(i)$\Rightarrow$(ii).] Let $y\in V$. Then $T(\tau)y\in V$.
By Lemma~\ref{lem:5.3} there exists a $w\in\MR(V,H)$ such that $f:=\dot{w}+\cA(\cdot)w\in L^2(0,\tau;H)$, $w(0)=0$ and $w(\tau)=T(\tau)y$. Let $u\in\MR(V,H)$ be such that $\dot{u}+\cA(\cdot)u=f$, $u(0)=\Phi u(\tau)$. Let $v=u-w$. Then $\dot{v}=\cA(\cdot)v = 0$ and $v(0)=u(0)=:x\in V$.
Hence $T(\tau)x=v(\tau)=u(\tau)-T(\tau)y$. Thus $\Phi T(\tau)x = x- \Phi T(\tau)y$.
Hence $(I-\Phi T(\tau))x = \Phi T(\tau)y$ and therefore $(I-\Phi T(\tau))(x+y)=y$.
So $(I-\Phi T(\tau))^{-1}y=x+y\in V$.

\item[(iii)$\Rightarrow$(i).] By Theorem~\ref{thm:5.1} one has $T(\tau)V\subset V$. Hence~(iii) implies that $\Phi T(\tau)V\subset V$. Since $(I-\Phi T(\tau))H=H$, the condition $(I-\Phi T(\tau))^{-1}V\subset V$ implies that $(I-\Phi S(\tau))V=V$. This implies~(i) by Theorem~\ref{thm:5.2}.\qedhere
\end{parenum}
\end{proof}

As an illustration, we show how the results can be applied to a semilinear problem.
In fact, maximal regularity in $H$ implies compactness properties which allow us to apply Schauder's or Schaefer's fixed point theorem.

Let $\cV = L^2(0,\tau;V)$ and $\cH=L^2(0,\tau;H)$. 
We continue to assume the conditions on the form $\form{a}$ from the beginning of this section.
\begin{thm}\label{thm:5.5}
Let $\KK=\RR$. Assume that the embedding $V\embed H$ is compact.
Assume that there exists an $\alpha_1>0$ such that $\form{a}(t,v,v)\ge\alpha_1\norm{v}_V^2$ for all $v\in V$ and $t\in[0,\tau]$.
Let $\Phi\in\Linop(H)$ with $\norm{\Phi}_{\Linop(H)}\le 1$.
Assume that $\Phi V\subset V$ and $(I-\Phi T(\tau))^{-1}V\subset V$.
Let $F\colon\cH\to\cH$ be continuous such that
\[
    \norm{F(v)}_{\cH} \le \alpha_2\norm{v}_{\cH} + \beta_1\quad\text{for all $v\in\cH$,}
\]
where $\beta_1>0$ and $\alpha_2\in(0,c_H^2\alpha_1)$.

Then there exists a $u\in\MR(V,H)$ such that $u(t)\in D(A(t))$ a.e., $A(\cdot)u\in\cH$ and
\[
    \left\{\begin{aligned}
        &\dot{u}+A(\cdot)u = F(u), \\
        &u(0) = \Phi u(\tau).
    \end{aligned}\right.
\]
\end{thm}
\begin{proof}
Consider the map $S\colon\cH\to\cH$ given by $Sv=u$, where $u\in\MR(V,H)$ is the unique solution of 
\[
    \left\{\begin{aligned}
        &\dot{u} + A(\cdot)u = F(v), \\
        &u(0) = \Phi u(\tau).
    \end{aligned}\right.
\]
The space
\begin{multline*}
    \MRa(V,H) := \{u\in H^1(0,\tau;H)\cap L^2(0,\tau;V) : u(t)\in D(A(t))\text{ a.e.} \\
        \text{and $t\mapsto A(t)u(t)\in L^2(0,\tau;H)$}\} 
\end{multline*}
is a Hilbert space for the norm
\[
    \norm{u}_{\MRa}^2 := \norm{\dot{u}}_\cH^2 + \norm{u}_\cV^2 + \norm{A(\cdot)u}_{L^2(0,\tau;H)}^2
\]
such that $\MRa(V,H)\embed C([0,\tau];V)$; see~\cite[Corollary~3.3]{ADLO14}. Moreover,
\[
    \MR_\Phi := \{u\in\MRa(V,H) : u(0)=\Phi u(\tau)\}
\]
is a closed subspace of $\MRa(V,H)$. The map $R\colon\MR_\Phi\to\cH$ given by $Ru = \dot{u} + A(\cdot)u$ is linear, continuous and bijective (by Proposition~\ref{prop:5.5new}).
Hence $R^{-1}\colon\cH\to\MR_\Phi$ is continuous.

By the Aubin--Lions theorem the embedding $j\colon\MR(V,H)\embed\cH$ is compact.
It follows that $S=j\circ R^{-1}\circ F\colon\cH\to\cH$ is continuous and compact.

In order to apply Schaefer's fixed point theorem~\cite[\S9.2.2 Theorem~4]{Eva98},
we have to show that the Schaefer set
\[
    \scrS := \{v\in\cH : \text{there exists $\lambda\in(0,1)$ such that $v=\lambda Sv$}\}
\]
is bounded in $\cH$.

Let $v\in\scrS$. Then $\dot{v}+A(\cdot) v = \lambda F(v)$, $v(0)=\Phi v(\tau)$.
Hence
\[
    \int_0^\tau \scalar{\dot{v}(t)}{v(t)}_H\dx[t] + \int_0^\tau \form{a}(t,v(t),v(t))\dx[t] = \lambda\scalar{F(v)}{v}_\cH.
\]
Since
\[
    \int_0^\tau\scalar{\dot{v}(t)}{v(t)}_H\dx[t] = \tfrac{1}{2}\norm{v(\tau)}_H^2 - \tfrac{1}{2}\norm{v(0)}_H^2 = \tfrac{1}{2}\norm{v(\tau)}_H^2 - \tfrac{1}{2}\norm{\Phi v(\tau)}_H^2\ge 0,
\]
it follows that
\[
    \alpha_1 c_H^2\norm{v}_\cH^2 \le \alpha_1\norm{v}_\cV^2 \le\norm{F(v)}_\cH\norm{v}_\cH.
\]
Thus
\[
    \alpha_1 c_H^2\norm{v}_\cH \le \norm{F(v)}_\cH \le \alpha_2\norm{v}_\cH + \beta_1.
\]
Hence $(\alpha_1 c_H^2-\alpha_2)\norm{v}_\cH\le \beta_1$.
Now Schaefer's fixed point theorem implies that there exists $u\in\cH$ such that $u=Su$.
Then $u$ solves the problem as desired.
\end{proof}

\begin{example}
Let $\KK=\RR$ and $H=L^2(\Omega)$, where $\Omega\subset\RR^d$ is open and bounded.
Let $V=H^1_0(\Omega)$ or $V=H^1(\Omega)$. In the latter case we in addition assume that $\Omega$ has Lipschitz boundary.
Then $V$ is compactly embedded in $L^2(\Omega)$.

Let $a_{kl}=a_{lk}\in L^\infty((0,\tau)\times\Omega)$ for all $k,l=1,\dots,d$ and suppose there exists an $\alpha>0$ such that
\[
    \sum_{k,l=1}^d a_{kl}(t,x)\xi_l\xi_k\ge\alpha\abs{\xi}^2
\]
for all $t\in [0,\tau]$, $x\in\Omega$ and $\xi\in\RR^d$.
Furthermore, let $c_k\in L^\infty((0,\tau)\times\Omega)$ for $k=0,\dots, d$.
We define the non-autonomous form $\form{a}\colon [0,\tau]\times V\times V \to\RR$ by
\[
    \form{a}(t,v,w) := \int_\Omega\Bigl(\sum_{k,l=1}^d a_{kl}(t,x)(\partial_l v)\partial_k w + \sum_{k=1}^d c_k(t,x)(\partial_k v)w + c_0(t,x)v w\Bigr)\dx.
\]
We assume that there exists an $\alpha_1>0$ such that
\[
    \form{a}(t,v,v)\ge\alpha_1\norm{v}_V^2\quad\text{for all $v\in V$, $t\in [0,\tau]$.}
\]
This can be arranged by replacing $c_0$ by $c_0+\omega$ for some $\omega\ge0$ big enough.

Let $F\colon\MR(V,H)\to\cH$ be given by
\[
    (Fv)(t,x) = g(t,x,v(t,x)),
\]
where $g\colon\RR\times\RR^d\times\RR\to\RR$ is continuous and satisfies
\[
    \abs{g(t,x,y)}^2 \le\alpha_2^2\abs{y}^2 + (h(x))^2\quad\text{for all $t\in [0,\tau]$, $x\in\RR^d$, $y\in\RR$,}
\]
where $\alpha_2\in (0,c_H^2\alpha_1)$ and $0\le h\in L^\infty(\RR^d)$.
Denote by $A(t)$ the operator in $H$ associated with $\form{a}(t,\cdot,\cdot)$.
Let $\Phi\in\Linop(L^2(\Omega))$ with $\norm{\Phi}_{\Linop(L^2(\Omega))}\le 1$ be such that $\Phi V\subset V$
and $(I-\Phi T(\tau))^{-1} V\subset V$.
Then Theorem~\ref{thm:5.5} gives the following result.
There exists a $u\in\MR(V,H)$ such that $u(t)\in D(A(t))$ a.e., $\dot{u}(t)+A(t)u(t)=g(t,\cdot, u(t,\cdot))$ in $L^2(\Omega)$ for a.e. $t\in(0,\tau)$ and $u(0)=\Phi u(\tau)$.
\end{example}

\section{Failure of maximal regularity in \texorpdfstring{$H$}{H}}\label{sec:7}

In this section we adopt the setting of Lion's theorem. We give an example in which maximal regularity in $H$ fails for any linear $\Phi\colon H\to H$, even if the problem is well-posed and $\norm{\Phi}_{\Linop(H)}\le 1$. It is a modification of Dier's counterexample~\cite[Section~5]{ADF17}.

\begin{thm}\label{thm:6.1}
There exist separable Hilbert spaces $V,H$ such that $V\dembed H$ and there exists a non-autonomous coercive form $\form{a}\colon[0,\tau]\times V\times V\to\CC$
such that the following holds.
Let $\cA(t)\in\Linop(V,V')$ be given by
$\langle\cA(t)v,w\rangle_{V',V} := \form{a}(t,v,w)$.
Let $\Phi\in\Linop(H)$.
Then there exists a $g\in L^2(0,\tau;H)$ such that the solution $u\in\MR(V,V')$ of
\[
    \left\{
    \begin{aligned}
    &\dot{u}+\cA(\cdot) u = g,\\
    &u(0)=\Phi u(\tau),
    \end{aligned}
    \right.
\]
is not in $\MR(V,H)$.
\end{thm}

We will use the following auxiliary results.
\begin{lem}\label{lem:6.2}
Let $-\infty<a<b<\infty$. For each $x\in H$ there exists $w\in H^1(a,b;V')\cap L^2(a,b;V)$ such that $w(a)=x$.
\end{lem}
\begin{proof}
By~\cite[Proposition~1.2.10]{Lun95} the real interpolation space can be described as a trace space in the following way:
\[
    (V',V)_{\frac{1}{2},2} := \{w(a) : w\in H^1(a,b;V')\cap L^2(a,b;V)\}
\]
Since $V',V$ are Hilbert spaces, $(V',V)_{\frac{1}{2},2} = [V',V]_{\frac{1}{2}}$ (the complex interpolation space) and the latter coincides with $H$.
\end{proof}

\begin{lem}\label{lem:6.3}
Let $V,H$ be Hilbert spaces, $V\dembed H$, $\form{b}\colon V\times V\to\CC$ a continuous, coercive sesquilinear form. Let $B$ be the operator in $H$ associated with $\form{b}$. Then
\[
    D(B^{1/2}) = [H,D(B)]_{\frac{1}{2}} = \{w(0) : w\in H^1(0,1;H)\cap L^2(0,1;D(B))\}.
\]
\end{lem}
\begin{proof}
The operator $B$ has BIP, which implies the first equality by~\cite[Theorem~6.6.9]{Haa06}.
Since $[H,D(B)]_{\frac{1}{2}} = (H,D(B))_{\frac{1}{2},2}$ by~\cite[Theorem~1.18.10]{Tri78},
the second follows from~\cite[Proposition~1.2.10]{Lun95}.
\end{proof}

\begin{proof}[Proof of Theorem~\ref{thm:6.1}]
Let $V,H$ be separable Hilbert spaces such that $V\dembed H\dembed V'$, and $\form{b}\colon V\times V\to\CC$ be a continuous, coercive form such that $D(B^{1/2})\not\subset V$,
where $B$ is the operator in $H$ associated with $\form{b}$.
Such a form exists by a result of McIntosh~\cite{McI72}.

Define the form $\form{c}\colon V\times V\to\CC$ by $\form{c}(v,w) := \frac{1}{2}(\form{b}(v,w)+\conj{\form{b}(w,v)})$, let $\cC\in\Linop(V,V')$ be given by $\langle\cC v,w\rangle_{V',V}=\form{c}(w,v)$ and let $C$ be the operator in $H$ associated with $\form{c}$ (i.e., $C$ is the part of $\cC$ in $H$).
Then $D(C^{1/2})=V$ since $\form{c}$ is symmetric.
Define the non-autonomous form $\form{a}\colon [0,\tau]\times V\times V \to\CC$ by
\[
    \form{a}(t,v,w) := \begin{cases}
        \form{b}(v,w) & \text{if $t\in[0,1)$,} \\
        \form{c}(v,w) & \text{if $\in[1,2]$.}
    \end{cases}
\]
Then the operator $A(t)$ associated with $\form{a}(t,\cdot,\cdot)$ in $H$ is given by
\[
    A(t) = \begin{cases}
        B & \text{if $t\in[0,1)$,} \\
        C & \text{if $t\in[1,2]$.} 
    \end{cases}
\]
Choose $x_0\in D(B^{1/2})\setminus V$.
Then by Lemma~\ref{lem:6.3} there exists $v\in H^1(0,1;H)\cap L^2(0,1;D(B))$ such that $v(0)=0$, $v(1)=x_0$.

Moreover, by Lemma~\ref{lem:6.2} there exists $w_1\in H^1(1,2;V')\cap L^2(1,2;V)$ such that $w_1(1)=x_0$.
By Lions' theorem there exists $w_2\in H^1(1,2;V')\cap L^2(1,2;V)$ such that $\dot{w}_2+\cC w_2 = \dot{w}_1 + \cC w_1$, $w_2(1)=0$. Then $w=w_1-w_2\in H^1(1,2;V')\cap L^2(1,2;V)$, $\dot{w}+\cC w=0$, $w(1)=x_0$. Since $x_0\notin V$, one has $w\notin H^1(1,2;H)$, see Theorem~\ref{thm:5.1}\,(a). 
Notice that the form $\form{c}$ is symmetric and constant on $(1,2)$, so the hypotheses of Theorem~\ref{thm:5.1} are satisfied for this form.

Let $\psi\in C^\infty([1,2];V)$ be such that $\psi(1)=1$, $\psi(2)=0$. Then $\psi w\in H^1(1,2;V')\cap L^2(1,2;V)$ and $(\psi w)' + \cC(\psi w) = \dot{\psi}w + \psi\dot{w} + \psi\cC w = \dot{\psi} w \in L^2(1,2;V)\subset L^2(1,2;H)$.
Let
\[
    u(t) := \begin{cases}
        v(t) & \text{if $t\in[0,1)$,} \\
        \psi(t)w(t) & \text{if $t\in[1,2]$.} 
    \end{cases}
\]
Then $u\in H^1(0,2;V')\cap L^2(0,2;V)$ (since $v(1)=x_0=\psi(1)w(1)$) and $\dot{u}+\cA(\cdot)u=g\in L^2(0,2;H)$, where
\[
    g(t) = \begin{cases}
        \dot{v}+\cA(\cdot) v & \text{on $[0,1)$,} \\
        \dot{\psi}w & \text{on $[1,2]$.} 
    \end{cases}
\]
Moreover $u(0)=u(2)=0$. Thus $u\in H^1(0,2;V')\cap L^2(0,2;V)$ is the unique solution of
\[
    \left\{
    \begin{aligned}
    &\dot{u} + \cA(\cdot) u = g,\\
    &u(0) = \Phi u(2).
    \end{aligned}
    \right.
\]
But $u\notin H^1(0,2;H)$.
\end{proof}

\end{document}